\documentclass[preprint,12pt]{article} 

\usepackage{amsmath,amsthm,amssymb}
\usepackage{hyperref}
\usepackage{mathrsfs}  
\usepackage{dsfont}  
\usepackage{enumerate}  
\usepackage{graphicx}
\allowdisplaybreaks[4]

\newtheorem{lemma}{Lemma}[section]
\newtheorem{theorem}{Theorem}[section]
\newtheorem{corollary}{Corollary}[section]

\newtheorem{remark}{Remark}[section]

\numberwithin{equation}{section}

\newcommand{\sref}[1]{\textsl{Section~\ref{#1}}}
\newcommand{\ssref}[1]{\textsl{Subsection~\ref{#1}}}
\newcommand{\tref}[1]{\textsl{Theorem~\ref{#1}}}
\newcommand{\lref}[1]{\textsl{Lemma~\ref{#1}}}
\newcommand{\cref}[1]{\textsl{Corollary~\ref{#1}}}

\newcommand{\rref}[1]{\textsl{Remark~\ref{#1}}}
\newcommand{\eref}[1]{\textsc{(\ref{#1})}}

\newcommand{\norm}[1]{\left\Vert#1\right\Vert}
\newcommand{\set}[1]{\left\{#1\right\}}

\newcommand{\R}{\mathbb R}
\newcommand{\Z}{\mathbb Z}

\newcommand{\ve}{\varepsilon}

\newcommand{\ol}{\overline}
\newcommand{\ul}{\underline}
\newcommand{\p}{\partial}
\newcommand{\dps}{\displaystyle}
\newcommand{\wt}{\widetilde}
\newcommand{\wh}{\widehat}
\newcommand{\ra}{\rightarrow}
\newcommand{\Ra}{\Rightarrow}
\newcommand{\Lra}{\Leftrightarrow}
\newcommand{\Llra}{\Longleftrightarrow}

\begin{document}


\title{On the exterior Dirichlet problem
for special Lagrangian equations\footnotemark[1]}


\author{Zhisu Li\footnotemark[2]}

\renewcommand{\thefootnote}{\fnsymbol{footnote}}

\footnotetext[1]{This research is supported by NSFC.11671316.}
\footnotetext[2]{School of Mathematics and Statistics,
Xi'an Jiaotong University, Xi'an 710049, China;\\
\href{mailto:lizhisu@stu.xjtu.edu.cn}{lizhisu@stu.xjtu.edu.cn}.}

\date{\today}


\maketitle

\begin{abstract}
In this paper, we establish
the existence and uniqueness theorem
of the exterior Dirichlet problem
for special Lagrangian equations
with prescribed asymptotic behavior at infinity.
\end{abstract}





\section{Introduction}

Let $D$ be a bounded domain in $\R^n$ ($n\geq3$).
We consider in this paper the Dirichlet problem for
the special Lagrangian equation
\begin{equation}\label{eqn.sle}
\sum_{i=1}^{n}\arctan\lambda_i(D^2u)=\Theta
\end{equation}
in the exterior domain $\R^n\setminus\ol{D}$,
where $\lambda_i(D^2u)$'s denote the eigenvalues
of the Hessian matrix $D^2u$,
and $\Theta$ is a constant such that
$(n-2)\pi/2\leq|\Theta|<n\pi/2$.

The special Lagrangian equation \eref{eqn.sle}
originates in the calibrated geometry \cite{HL82},
and plays an important role in the study of string theory
(see for example \cite{CMMS04}).
The left hand side of the equation \eref{eqn.sle}
indeed stands for the argument of the complex number
$(1+\sqrt{-1}\lambda_1(D^2u))...(1+\sqrt{-1}\lambda_n(D^2u))$,
which is usually called \emph{phase} or \emph{Lagrangian phase}.
When the phase is constant, the gradient graph
$\set{(x,Du(x))}\subset\R^n\times\R^n=\mathbb{C}^n$
then is called \emph{special Lagrangian}.
One can prove that $\set{(x,Du(x))}$ is special Lagrangian,
i.e., \eref{eqn.sle} holds for some constant
$\Theta\in(-n\pi/2,n\pi/2)$, if and only if
it is a volume minimizing minimal surface in $\R^n\times\R^n$.
In the literature, the Lagrangian phase $(n-2)\pi/2$
is usually called \emph{critical}, since the level set
\[L_\Theta:=\set{\lambda\in\R^n\bigg{|}\sum_{i=1}^{n}\arctan\lambda_i=\Theta}\]
is convex only if $|\Theta|\geq(n-2)\pi/2$ \cite[Lemma 2.1]{Yuan06}.
So the main results in this paper are concerning
the special Lagrangian equation \eref{eqn.sle}
with critical and supercritical phases.

Recently, we proved in \cite{LLY17} that
any smooth solution $u$ of \eref{eqn.sle}
with supercritical phase in the exterior domain
must tend to a quadratic polynomial $Q$ at infinity,
and satisfy
\begin{equation}\label{eqn.u=Q+O}
u(x)=Q(x)+O\left(\frac{1}{|x|^{n-2}}\right)
~(|x|\ra+\infty).
\end{equation}
That is, there exist
$A\in\R^{n\times n},b\in\R^n$ and $c\in\R$
such that
\begin{equation}\label{eqn.abc-n}
\limsup_{|x|\ra+\infty}|x|^{n-2}
\left|u(x)-\left(\frac{1}{2}x^{T}Ax+b^Tx+c\right)\right|<\infty.
\end{equation}
Thus for any given solution of \eref{eqn.sle} in the exterior domain,
we know definitely that
it obeys the above asymptotic behavior \eref{eqn.abc-n} at infinity.
So the converse problem arises, that is,
with any such prescribed asymptotic behavior at infinity,
whether there exists a unique solution of the Dirichlet problem
of the special Lagrangian equation \eref{eqn.sle} in the exterior domain?
The answer is yes (at least partially), and we will mainly focus on
this exterior Dirichlet problem in this paper.

To deal with the exterior Dirichlet problem,
as a lot of the previous researches suggested
(see \cite{CL03,DB11,BLL14,LB14}),
Perron's method is often adopted, and the key point of which is
to construct some appropriate subsolutions of the equation.
For this purpose, we turn to study the following algebraic form
of the special Lagrangian equation \eref{eqn.sle}:
\begin{equation}\label{eqn.sle-af}
\cos\Theta\sum_{0\leq 2k+1\leq n}(-1)^k\sigma_{2k+1}(\lambda(D^2u))
-\sin\Theta\sum_{0\leq 2k\leq n}(-1)^k\sigma_{2k}(\lambda(D^2u))=0,
\end{equation}
where $\sigma_k(\lambda)$'s are the elementary symmetric polynomials
with respect to $\lambda$, defined by
\[\sigma_0(\lambda)\equiv1
\quad\mathrm{and}\quad
\sigma_k(\lambda):=\sum_{1\leq s_1<s_2<...<s_k\leq n}
\lambda_{s_1}\lambda_{s_2}...\lambda_{s_k}
~(\forall 1\leq k\leq n).\]
Note that the solution of \eref{eqn.sle}
is always a solution of \eref{eqn.sle-af},
but the converse is not always true.
Our strategy is to construct
some proper subsolutions of \eref{eqn.sle-af},
and then come back to show that
they are exactly the desired subsolutions of \eref{eqn.sle}.
The techniques used here
to construct subsolutions of \eref{eqn.sle-af}
are partially inherited from our previous paper \cite{LL16}
concerning the exterior Dirichlet problem
for the Hessian quotient equations,
but are much harder than those and have been largely extended.

We would like to remark that, the rigidity theorems
for special Lagrangian equations on the whole space
have been fully studied by Prof. Y. Yuan, see for instance
\cite{Yuan02,Yuan06,WY08};
for more on the special Lagrangian equations,
we refer the readers to
\cite{HL82,Fu98,CWY09,WY14} and the references therein.

For the results concerning the exterior Dirichlet problems
of the Monge-Amp\`{e}re equations, of the Hessian equations,
and of the Hessian quotient equations,
see for example \cite{CL03,DB11,LD12,BLL14,LB14,LL16}
and the references therein.

\bigskip

Define
\begin{equation}\label{eqn.A0def}
\mathscr{A}^0_\Theta
:=\set{A\in S(n)\bigg{|}\lambda(A)\in\Gamma^+\cup\left(-\Gamma^+\right),
\sum_{i=1}^{n}\arctan\lambda_i(A)=\Theta},
\end{equation}
and
\begin{equation}\label{eqn.Adef}
\mathscr{A}_\Theta
:=\set{A\in\mathscr{A}^0_\Theta\bigg{|}
m\big(\Theta,\lambda(A)\big)>2},
\end{equation}
where
\begin{enumerate}[\quad(1)]
\item
$S(n)$ denotes the linear space of symmetric $n\times n$ real matrices;

\item
$\Gamma^+$ denotes the \emph{positive cone}
\[\Gamma^+:=\set{\lambda\in\R^n|\lambda_i>0,\forall i=1,2,...,n};\]

\item
$m(\Theta,\lambda)$ is a quantity introduced by us,
which plays an important role in the study of
the special Lagrangian equations.
For some special matrices $A\in\mathscr{A}_\Theta$,
for example, $A=\tan(\Theta/n)\,I_n$,
we have $m(\Theta,\lambda(A))=n$.
The specific definition of $m(\Theta,\lambda)$
will be given in \eref{eqn.mdef} in \ssref{subsec.XYZ},
when all the necessary preparation is done.
\end{enumerate}
Then the main results of this paper can be stated as the following theorems.
\begin{theorem}\label{thm.sle-Theta>0}
Let $D$ be a bounded strictly convex domain
in $\R^n$, $n\geq 3$, $\p D\in C^2$
and let $\varphi\in C^2(\p D)$.
Then for any given $A\in\mathscr{A}_\Theta$
with $(n-2)\pi/2\leq\Theta<n\pi/2$, and
any given $b\in \R^n$,
there exists a constant ${c_\ast}$
depending only on $n,D,\Theta,A,b$
and $\norm{\varphi}_{C^2(\p D)}$,
such that for every $c\geq{c_\ast}$,
there exists a unique viscosity solution
$u\in C^0(\R^n\setminus D)$ of
\begin{equation}\label{eqn.sle-abc}
\left\{
\begin{aligned}
\dps \sum_{i=1}^{n}\arctan\lambda_i&\left(D^2u\right)=\Theta
\quad\mathrm{in}~\R^n\setminus\ol{D},\\
\dps &\qquad~u=\varphi\quad\mathrm{on}~\p D,\\[0.2cm]
\dps \limsup_{|x|\ra+\infty}|x|^{m-2}
&\left|u(x)-\left(\frac{1}{2}x^{T}Ax+b^Tx+c\right)\right|<\infty,
\end{aligned}
\right.
\end{equation}
where $m\in(2,n]$ is a constant depending
only on $n,\Theta$ and $\lambda(A)$,
which actually can be taken as $m(\Theta,\lambda(A))$.
\end{theorem}

\begin{theorem}\label{thm.sle-Theta<0}
Let $D$ be a bounded strictly convex domain
in $\R^n$, $n\geq 3$, $\p D\in C^2$
and let $\varphi\in C^2(\p D)$.
Then for any given $A\in\mathscr{A}_\Theta$
with $-n\pi/2<\Theta\leq-(n-2)\pi/2$, and
any given $b\in \R^n$,
there exists a constant ${c^\ast}$
depending only on $n,D,\Theta,A,b$
and $\norm{\varphi}_{C^2(\p D)}$,
such that for every $c\leq{c^\ast}$,
there exists a unique viscosity solution
$u\in C^0(\R^n\setminus D)$ of \eref{eqn.sle-abc}, 
where $m\in(2,n]$ is a constant depending
only on $n,\Theta$ and $\lambda(A)$,
which actually can be taken as $m(-\Theta,-\lambda(A))$.
\end{theorem}

\begin{remark}
\begin{enumerate}[(1)]
\item
By symmetry, letting $\wt u=-u$, we see easily that
\tref{thm.sle-Theta>0} and \tref{thm.sle-Theta<0} are equivalent.
Thus we need only to consider $a=\lambda(A)\in\Gamma^+$
and prove \tref{thm.sle-Theta>0} in the rest of this paper.

\item
$\mathscr{A}_\Theta$ is a large class of matrices
containing $\tan(\Theta/n)\,I_n$,
which will be clear in \ssref{subsec.XYZ}.
For some particular $n$ and $\Theta$ we will see that
$\mathscr{A}_\Theta=\mathscr{A}^0_\Theta$,
which is the best possible situation we can hope.
For example, if $n=3,4$, $\Theta=\pm\pi$,
we have the special Lagrangian equation in the algebraic form
$\sigma_3(\lambda(D^2u))=\sigma_1(\lambda(D^2u))$
(in three dimension, this indeed is $\det(D^2u)=\Delta u$).
Then the $A\in\mathscr{A}_\Theta$
in \tref{thm.sle-Theta>0} and \tref{thm.sle-Theta<0},
can be chosen as any matrix $A\in S(n)$ such that
$\lambda(A)\in L_\Theta\cap(\Gamma^+\cup(-\Gamma^+))$,
i.e., any $A\in\mathscr{A}^0_\Theta$.
This result is obtained for the first time
in our previous paper \cite{LL16},
since $\sigma_3(\lambda(D^2u))=\sigma_1(\lambda(D^2u))$
is also a Hessian quotient equation.

\item
The restrictions on the Lagrangian phase $\Theta$
to be critical or supercritical are subtle and technical.
We believe that they can be discarded more or less,
since, from \cite{Yuan02}, we know that
any convex solution of \eref{eqn.sle}
(no restriction on the Lagrangian phase) in $\R^n$
must be a quadratic polynomial.

\item
As far as we know, there are only two papers
concerning the exterior Dirichlet problem
for special Lagrangian equations,
but both are highly restricted.
For example,
\cite{BL13} considered the problem in the cases that
$n\leq4$, $\Theta=\pm\pi$ and $A=\tan(\Theta/n)\,I_n$,
which are just the Hessian quotient equations;
\cite{LB14} considered the problem in the cases that
$n\geq3$, $|\Theta|>(n-1)\pi/2$ and $A=\tan(\Theta/n)\,I_n$.
We remark that the results appeared
in \cite{BL13} and \cite{LB14}
can all be recovered by our \tref{thm.sle-Theta>0}
and \tref{thm.sle-Theta<0} as special cases.
\qed
\end{enumerate}
\end{remark}

\medskip

The paper is organized as follows.
In \sref{sec.prel},
we first give some specific definitions of terminologies
which are standard in the literature,
and a list of notations introduced only by us in this paper.
Then we collect in \ssref{subsec.prelem} some well known lemmas
which are mostly used in \sref{sec.pmaint}.
In \sref{sec.Q&P},
we first introduce quantities $\Xi_k,\ul\xi_k,\ol\xi_k$
and investigate their properties in \ssref{subsec.xi}.
Then, in \ssref{subsec.XYZ}, we study polynomials
$X,Y,Z,\wh X,\wh Y$ and $\wh Z$
related to the algebraic special Lagrangian equation \eref{eqn.sle-af},
which play fundamental role in this paper.
\sref{sec.pmaint} is devoted to
the proof of the main theorem (\tref{thm.sle-Theta>0}).
To do so, we start in \ssref{subsec.csubsol}
to construct some appropriate subsolutions of
the special Lagrangian equation \eref{eqn.sle},
by solving some proper ordinary differential equation.
Then, after reducing \tref{thm.sle-Theta>0} to \lref{lem.sle}
by normalization in \ssref{subsec.lemsle>thmsle},
we prove \lref{lem.sle} in \ssref{subsec.pr-lem-sle}
by applying the Perron's method to
the subsolutions we constructed in \ssref{subsec.csubsol}.
The last section, \sref{sec.appendix}, is an appendix,
which is devoted to prove \lref{lem.XhY-YhX=SS}
asserted by us in \ssref{subsec.XYZ}.

\section{Preliminary}\label{sec.prel}

\subsection{Notation}\label{subsec.nott}

In this paper,
$S(n)$ denotes the linear space of symmetric $n\times n$ real matrices,
and $I_n$ denotes the identity matrix.

For any $M\in S(n)$,
if $m_1,m_2,...,m_n$ are the eigenvalues of $M$
(usually, the assumption $m_1\leq m_2\leq...\leq m_n$
is added for convenience),
we will denote this fact briefly by $\lambda(M)=(m_1,m_2,...,m_n)$
and call $\lambda(M)$ the eigenvalue vector of $M$.
For convenience,
we write often $\mathds{1}:=\lambda(I_n)=(1,1,...,1)$.

For $A\in S(n)$ and $\rho>0$, we denote by
\[
E_\rho:=\left\{x\in\R^n\big{|}x^TAx<\rho^2\right\}
=\left\{x\in\R^n\big{|}r_A(x)<\rho\right\}
\]
the ellipsoid of size $\rho$ with respect to $A$,
where we set $r_A(x):=\sqrt{x^TAx}$.

\medskip

For any $p\in\R^n$, we write
\[\sigma_k(p):=\sum_{1\leq s_1<s_2<...<s_k\leq n}
p_{s_1}p_{s_2}...p_{s_k}\quad(\forall 1\leq k\leq n)\]
as the $k$-th elementary symmetric function of $p$.
Meanwhile, we will adopt the conventions that
$\sigma_{-1}(p)\equiv0$, $\sigma_0(p)\equiv1$
and $\sigma_k(p)\equiv0$, $\forall k\geq n+1$;
and we will also define
\[\sigma_{k;i}(p):=\left(\sigma_{k}(\lambda)
\big{|}_{\lambda_i=0}\right)\Big{|}_{\lambda=p}
=\sigma_{k}\left(p_1,p_2,...,\widehat{p_i},...,p_n\right)\]
for any $-1\leq k\leq n$ and any $1\leq i\leq n$,
and similarly
\[\sigma_{k;i,j}(p):=\left(\sigma_{k}(\lambda)
\big{|}_{\lambda_i=\lambda_j=0}\right)\Big{|}_{\lambda=p}
=\sigma_{k}\left(p_1,p_2,...,\widehat{p_i},
...,\widehat{p_j},...,p_n\right)\]
for any $-1\leq k\leq n$
and any $1\leq i,j\leq n$, $i\neq j$, for convenience.

\medskip

For the reader's convenience,
we give the following list of notations
which are only introduced by us in this paper.
\begin{enumerate}[\qquad$\diamond$]
\item
$\Xi_k=\Xi_k(a,x)$, $\ul\xi_k=\ul\xi_k(a)$,
$\ol\xi_k=\ol\xi_k(a)$
----- see \eref{eqn.Xikdef},
\eref{eqn.olxikdef}, \eref{eqn.olxikdef};

\item
$c_k=c_k(\Theta)$, $\xi_k=\xi_k(\Theta,a)$
----- see \eref{eqn.ckdef}, \eref{eqn.xikdef};

\item
$H(\lambda):=\sum\arctan\lambda_i$
----- see \eref{eqn.Hdef};

$X:=1-\sigma_2+\sigma_4-...$
----- see \eref{eqn.Xdef};

$Y:=\sigma_1-\sigma_3+\sigma_5-...$
----- see \eref{eqn.Ydef};

$\wh X:=-2\sigma_2+4\sigma_4-...$
----- see \eref{eqn.hXdef};

$\wh Y:=\sigma_1-3\sigma_3+5\sigma_5-...$
----- see \eref{eqn.hYdef};

$Z:=\cos\Theta\;Y-\sin\Theta\;X=\sum{c_k\sigma_k}$
----- see \eref{eqn.Zdef};

$\wh Z:=\cos\Theta\;\wh Y-\sin\Theta\;\wh X=\sum{kc_k\sigma_k}$
----- see \eref{eqn.hZdef};

\item
$N=N(n)$ ----- see \eref{eqn.Ndef};

$m(\Theta,a)$ ----- see \eref{eqn.mdef},
\eref{eqn.frac-K1=1/m}, \eref{eqn.mdlnr};

$\mathscr{A}^0_\Theta$, $\mathscr{A}_\Theta$
----- see \eref{eqn.A0def}, \eref{eqn.Adef};

\item
$\psi=\psi(r)=\phi'(r)/r$, $\phi=\phi(r)$, $\Phi=\Phi(x)=\phi(r_A(x))$
----- see \lref{lem.psi}, \eref{eqn.phidef}, \eref{eqn.Phidef}.
\end{enumerate}

The reasons for why we define them will be clear in future sections.
We just give a few comments about them here now.
To establish the existence of the solution
of the special Lagrangian equation \eref{eqn.sle}
by the Perron's method, the key point is
to construct some appropriate subsolutions
of its algebraic form \eref{eqn.sle-af}.
Since \eref{eqn.sle-af} is a fully nonlinear equation
consisting of polynomials with respect to the eigenvalues
of the Hessian matrix $D^2u$
(i.e., polynomials $\sigma_k(\lambda(D^2u))$, $k=1,2,...,n$)
of different order of homogeneities,
to solve it we need to strike a balance among them.
It will turn out to be clear that
the quantities $\Xi_k,\ul\xi_k,\ol\xi_k$ and $\xi_k$
are very natural and perfectly fit for this purpose.
Roughly speaking, $\Xi_k$ originates from
the computation of $\sigma_k(\lambda(D^2\Phi(x)))$
where $\Phi(x)=\phi(r_A(x))$ is a generalized radially symmetric function
(see \lref{lem.skm} and the proof of \lref{lem.Phi-subsol});
$\ul\xi_k,\ol\xi_k$ and $\xi_k$ result from
the comparison among different $\sigma_k(\lambda)$'s
in the attempt to derive an ordinary differential equation
from the original equation \eref{eqn.sle-af}
(see the proof of \lref{lem.Phi-subsol});
$m(\Theta,a)$ arises in the process of solving
this ordinary differential equation
(see \eref{eqn.frac-K1=1/m} and \eref{eqn.mdlnr}
in the proof of \lref{lem.psi}),
so do the polynomials $X,Y,Z,\wh X,\wh Y,\wh Z$
and their coefficients $c_k$,
which offer us deeper understanding
of the algebraic special Lagrangian equation \eref{eqn.sle-af}
and play crucial roles in the construction of the subsolutions.
By $\Xi_k,\ul\xi_k,\ol\xi_k$ and $\xi_k$,
we get a good balance among different $\sigma_k(\lambda)$'s,
which can be viewed as being measured by $m(\Theta,a)$.
Furthermore, we will find that
$m(\Theta,a)$ has also some special meaning related to
the decay and asymptotic behavior of the solution
(see \lref{lem.psi}\textsl{-(ii)},
\cref{cor.mub} and \tref{thm.sle-Theta>0}).

\subsection{Some preliminary lemmas}\label{subsec.prelem}

In this subsection,
we collect some well known preliminary lemmas
which will be mainly used in \sref{sec.pmaint}.

\medskip

We first give a lemma to compute $\sigma_k(\lambda(M))$
with $M$ of certain type.
~If $\Phi(x):=\phi(r)$ with $\phi\in C^2$, $r=\sqrt{x^TAx}$,
$A\in S(n)\cap\Gamma^+$ and $a=\lambda(A)$
(we may call $\Phi$ a \emph{generalized radially symmetric function}
with respect to $A$, according to \cite{BLL14}),
one can conclude that
\[\p_{ij}\Phi(x)
=\frac{\phi'(r)}{r}a_i\delta_{ij}+
\frac{\phi''(r)-\frac{\phi'(r)}{r}}{r^2}(a_ix_i)(a_jx_j),
~\forall 1\leq i,j\leq n,\]
provided $A$ is normalized to a diagonal matrix
(see \ssref{subsec.lemsle>thmsle}
and the proof of \lref{lem.Phi-subsol} for details).
As far as we know, there is generally no explicit formula
for $\lambda(D^2\Phi(x))$ of this type,
but luckily we have a method
to calculate $\sigma_k\left(\lambda(D^2\Phi(x))\right)$
for each $1\leq k\leq n$,
which can be represented as the following lemma.
\begin{lemma}\label{lem.skm}
If $M=\left(p_i\delta_{ij}+s q_iq_j\right)_{n\times n}$
with $p,q\in\R^n$ and $s\in\R$,
then
\[\sigma_k\left(\lambda(M)\right)
=\sigma_k(p)+s\sum_{i=1}^n\sigma_{k-1;i}(p)q_i^2,
~\forall 1\leq k\leq n.\]
\end{lemma}
\begin{proof}
See \cite{BLL14}.
\end{proof}

\medskip

To process information on the boundary
we need the following lemma.
\begin{lemma}\label{lem.Qxi}
Let $D$ be a bounded strictly convex domain
of $\R^n$, $n\geq 2$, $\p D\in C^2$,
$\varphi\in C^0(D)\cap C^2(\p{D})$
and let $A\in S(n)$, $\det{A}\neq0$.
Then there exists a constant $K>0$ depending only on $n$,
$\mbox{\emph{diam}}\,D$, the convexity of $D$,
$\norm{\varphi}_{C^2(\ol{D})}$,
the $C^2$ norm of $\p D$ and the upper bound of $A$,
such that for any $\xi\in\p D$,
there exists $\bar{x}(\xi)\in\R^n$ satisfying
\[\left|\bar{x}(\xi)\right|\leq K
\quad \mbox{and} \quad Q_\xi(x)<\varphi(x),
~\forall x\in \ol{D}\setminus\{\xi\},\]
where
\[Q_\xi(x):=\frac{1}{2}\left(x-\bar{x}(\xi)\right)^TA\left(x-\bar{x}(\xi)\right)
-\frac{1}{2}\left(\xi-\bar{x}(\xi)\right)^TA\left(\xi-\bar{x}(\xi)\right)
+\varphi(\xi),~\forall x\in\R^n.\]
\end{lemma}
\begin{proof}
See \cite{CL03} or \cite{BLL14}.
\end{proof}

\begin{remark}\label{rmk.Qxi}
It is easy to check that
$Q_\xi$ satisfy the following properties.
\begin{enumerate}[\quad(1)]
\item $Q_\xi\leq\varphi$ on $\ol{D}$
and $Q_\xi(\xi)=\varphi(\xi)$.

\item If $A\in\mathscr{A}^0_\Theta$, then
\[\sum_{i=1}^{n}\arctan\lambda_i\left(D^2Q_\xi\right)=\Theta
\quad\mbox{in}~\R^n.\]

\item There exists $\bar{c}=\bar{c}(D,A,K)>0$ such that
\[Q_\xi(x)\leq\frac{1}{2}x^TAx+\bar{c},
\quad \forall x\in\p D,~\forall\xi\in\p D.\]
\end{enumerate}
\end{remark}

\medskip

Now we introduce the following well known lemmas
about the comparison principle and Perron's method.
These lemmas are adaptions of those appeared
in \cite{CNS85} \cite{Jen88} \cite{Ish89}
\cite{Urb90} and \cite{CIL92}.
For specific proof of them one may also consult
\cite{BLL14} and \cite{LB14}.
\begin{lemma}[Comparison principle]\label{lem.cp}
Assume $f\in C^1(\R^n)$ and
$f_{\lambda_i}(\lambda)>0$, $\forall \lambda\in\R^n$,
$\forall i=1,2,...,n$.
Let $\Omega\subset\R^n$ be a domain
and let $\ul{u},\ol{u}\in C^0(\ol\Omega)$ satisfying
\[f\left(\lambda\left(D^2\ul{u}\right)\right)\geq 0
\geq f\left(\lambda\left(D^2\ol{u}\right)\right)\]
in $\Omega$ in the viscosity sense.
Suppose $\ul{u}\leq \ol{u}$ on $\p\Omega$
(and additionally
\[\lim_{|x|\ra+\infty}\left(\ul{u}-\ol{u}\right)(x)=0\]
provided $\Omega$ is unbounded).
Then $\ul{u}\leq \ol{u}$ in $\Omega$.
\end{lemma}

\begin{lemma}[Perron's method]\label{lem.pm}
Assume $f\in C^1(\R^n)$ and
$f_{\lambda_i}(\lambda)>0$, $\forall \lambda\in\R^n$,
$\forall i=1,2,...,n$.
Let $\Omega\subset\R^n$ be a domain, $\varphi\in C^0(\p\Omega)$
and let $\ul{u},\ol{u}\in C^0(\ol\Omega)$ satisfying
\[f\left(\lambda\left(D^2\ul{u}\right)\right)\geq 0
\geq f\left(\lambda\left(D^2\ol{u}\right)\right)\]
in $\Omega$ in the viscosity sense.
Suppose $\ul{u}\leq \ol{u}$ in $\Omega$,
$\ul{u}=\varphi$ on $\p\Omega$
(and additionally
\[\lim_{|x|\ra+\infty}\left(\ul{u}-\ol{u}\right)(x)=0\]
provided $\Omega$ is unbounded).
Then
\begin{eqnarray*}
u(x)&:=&\sup\Big\{v(x)\big{|}v\in C^0(\Omega),~
\ul{u}\leq v\leq \ol{u}~\mbox{in}~\Omega,~
f\left(\lambda\left(D^2v\right)\right)\geq 0~\mbox{in}~\Omega\\
&~&\qquad\mbox{in the viscosity sense},~
v=\varphi~\mbox{on}~\p\Omega\Big\}
\end{eqnarray*}
is the unique viscosity solution of the Dirichlet problem
\[\left\{
\begin{aligned}
f\left(\lambda\left(D^2u\right)\right)=0 \qquad &\mbox{in}& \Omega,\\
u=\varphi \qquad &\mbox{on}& \p\Omega.
\end{aligned}\right.
\]
\end{lemma}

\begin{remark}\label{rmk.f=arctan}
Clearly, by letting
\[f(\lambda):=H(\lambda)-\Theta
\triangleq\sum_{i=1}^{n}\arctan\lambda_i-\Theta,\]
the special Lagrangian equation \eref{eqn.sle}
satisfies the above two lemmas.
\end{remark}

\section{Quantities and polynomials related to
the special Lagrangian equations}\label{sec.Q&P}

\subsection{Quantities $\Xi_k,\ul\xi_k,\ol\xi_k$ and their properties}\label{subsec.xi}

The quantities $\Xi_k,\ul\xi_k$ and $\ol\xi_k$
originate in the calculation of $\sigma_k(\lambda(D^2\Phi))$
and in the comparison among different $\sigma_k(\lambda(D^2\Phi))$'s
(see \lref{lem.skm} and the proof of \lref{lem.Phi-subsol}).
They were partially given for the first time
in our previous paper \cite{LL16}.
For completeness and convenience,
we introduce them again in this subsection.

For any $0\leq k\leq n$ and any $a\in\R^n\setminus\{0\}$, let
\begin{equation}\label{eqn.Xikdef}
\Xi_k:=\Xi_k(a,x)
:=\frac{\sum_{i=1}^n\sigma_{k-1;i}(a)a_i^2x_i^2}
{\sigma_k(a)\sum_{i=1}^{n}a_ix_i^2},
~\forall x\in\R^n\setminus\{0\},
\end{equation}
and define
\begin{equation}\label{eqn.olxikdef}
\ol{\xi}_k:=\ol{\xi}_k(a)
:=\sup_{x\in\R^n\setminus\{0\}}\Xi_k(a,x),
\end{equation}
and
\begin{equation}\label{eqn.ulxikdef}
\ul{\xi}_k:=\ul{\xi}_k(a)
:=\inf_{x\in\R^n\setminus\{0\}}\Xi_k(a,x).
\end{equation}

\bigskip

It is easy to see that
\[\ol{\xi}_k(\varrho a)=\ol{\xi}_k(a),
~\ul{\xi}_k(\varrho a)=\ul{\xi}_k(a),
~\forall\varrho\neq0,
~\forall a\in\R^n\setminus\{0\},
~\forall 0\leq k\leq n,\]
and
\begin{equation}\label{eqn.xikrho1=k/n}
\ul\xi_k(\varrho\mathds{1})=\frac{k}{n}
=\ol\xi_k(\varrho\mathds{1}),
~\forall \varrho\neq0,~\forall 0\leq k\leq n.
\end{equation}
Furthermore, we have the following lemma.
\begin{lemma}\label{lem.xik}
Suppose $a=(a_1,a_2,...,a_n)$
with $0<a_1\leq a_2\leq...\leq a_n$.
Then
\begin{equation}\label{eqn.uxknox}
0<\frac{a_1\sigma_{k-1;1}(a)}{\sigma_k(a)}=\ul{\xi}_k(a)
\leq\frac{k}{n}\leq\ol{\xi}_k(a)
=\frac{a_n\sigma_{k-1;n}(a)}{\sigma_k(a)}\leq 1,
~\forall 1\leq k\leq n;
\end{equation}
\begin{equation}\label{eqn.olxiin}
0=\ol{\xi}_0(a)<\frac{1}{n}\leq\frac{a_n}{\sigma_1(a)}
=\ol{\xi}_1(a)\leq\ol{\xi}_2(a)
\leq...\leq\ol{\xi}_{n-1}(a)<\ol{\xi}_n(a)=1;
\end{equation}
and
\begin{equation}\label{eqn.ulxiin}
0=\ul{\xi}_0(a)<\frac{a_1}{\sigma_1(a)}
=\ul{\xi}_1(a)\leq\ul{\xi}_2(a)
\leq...\leq\ul{\xi}_{n-1}(a)<\ul{\xi}_n(a)=1.
\end{equation}
Moreover,
\begin{equation}\label{eqn.xkkn}
\ul\xi_k(a)=\frac{k}{n}=\ol\xi_k(a)
\end{equation}
for some $1\leq k\leq n-1$,
if and only if $a=C\mathds{1}$ for some $C>0$.
\end{lemma}

\begin{proof}
($1^\circ$)
By the definitions of $\sigma_k(a)$ and $\sigma_{k;i}(a)$,
we see that
\begin{equation}\label{eqn.sk}
\sigma_k(a)=\sigma_{k;i}(a)+a_i\sigma_{k-1;i}(a),
~\forall 1\leq i\leq n;
\end{equation}
and
\[\sum_{i=1}^{n}\sigma_{k;i}(a)
=\frac{nC_{n-1}^k}{C_n^k}\sigma_k(a)
=(n-k)\sigma_k(a).\]
Hence we obtain
\begin{equation}\label{eqn.ksk}
\sum_{i=1}^{n}a_i\sigma_{k-1;i}(a)=k\sigma_k(a).
\end{equation}

Now we show that
\begin{equation}\label{eqn.aiski}
a_1\sigma_{k-1;1}(a)
\leq a_2\sigma_{k-1;2}(a)\leq...
\leq a_n\sigma_{k-1;n}(a).
\end{equation}
In fact, for any $i\neq j$,
similar to \eref{eqn.sk}, we have
\[a_i\sigma_{k-1;i}(a)
=a_i\left(\sigma_{k-1;i,j}(a)+a_j\sigma_{k-2;i,j}(a)\right),\]
\[a_j\sigma_{k-1;j}(a)
=a_j\left(\sigma_{k-1;i,j}(a)+a_i\sigma_{k-2;i,j}(a)\right),\]
and hence
\[a_i\sigma_{k-1;i}(a)-a_j\sigma_{k-1;j}(a)
=(a_i-a_j)\sigma_{k-1;i,j}(a).\]
Therefore, if $a_i\lessgtr a_j$, then
\begin{equation}\label{eqn.aslgeq}
a_i\sigma_{k-1;i}(a)\lessgtr a_j\sigma_{k-1;j}(a).
\end{equation}

By the definition of $\ol\xi_k$, we have
\begin{eqnarray*}
\ol{\xi}_k(a)
&=&\sup_{x\neq 0}\frac{\sum_{i=1}^n\sigma_{k-1;i}(a)a_i^2x_i^2}
{\sigma_k(a)\sum_{i=1}^{n}a_ix_i^2}\\
&\geq& \sup_{\substack{x_1=...=x_{n-1}=0,\\x_n\neq 0}}
\frac{\sum_{i=1}^n\sigma_{k-1;i}(a)a_i^2x_i^2}
{\sigma_k(a)\sum_{i=1}^{n}a_ix_i^2}\\
&=&\sup_{x_n\neq 0}\frac{\sigma_{k-1;n}(a)a_n^2x_n^2}
{\sigma_k(a)a_n x_n^2}\\
&=&\frac{a_n\sigma_{k-1;n}(a)}{\sigma_k(a)}
\end{eqnarray*}
and
\begin{eqnarray*}
\ol{\xi}_k(a)
&=&\sup_{x\neq0}\frac{\sum_{i=1}^n\sigma_{k-1;i}(a)a_i^2x_i^2}
{\sigma_k(a)\sum_{i=1}^{n}a_ix_i^2}\\
&\leq& \sup_{x\neq0}\frac{a_n\sigma_{k-1;n}(a)\sum_{i=1}^na_ix_i^2}
{\sigma_k(a)\sum_{i=1}^{n}a_ix_i^2}
\qquad\text{\big(according to \eref{eqn.aiski}\big)}\\
&=&\frac{a_n\sigma_{k-1;n}(a)}{\sigma_k(a)}.
\end{eqnarray*}
Hence we obtain
\begin{equation}\label{eqn.olxik}
\ol\xi_k(a)=\frac{a_n\sigma_{k-1;n}(a)}{\sigma_k(a)}.
\end{equation}
Similarly
\begin{equation}\label{eqn.ulxik}
\ul\xi_k(a)=\frac{a_1\sigma_{k-1;1}(a)}{\sigma_k(a)}.
\end{equation}

From \eref{eqn.ksk}, we have
\[\sum_{i=1}^{n}\frac{a_i\sigma_{k-1;i}(a)}{\sigma_k(a)}=k.\]
Combining this with \eref{eqn.aiski},
\eref{eqn.olxik} and \eref{eqn.ulxik},
we deduce that
\[\ul\xi_k(a)\leq\frac{k}{n}\leq\ol\xi_k(a).\]
Thus the proof of \eref{eqn.uxknox} is complete,
and \eref{eqn.xkkn} is also clear in view of \eref{eqn.aslgeq}.

\medskip

($2^\circ$)
Since it follows from \eref{eqn.sk} that
\[a_i\sigma_{k-1;i}(a)<\sigma_k(a),
~\forall 1\leq i\leq n,
~\forall 1\leq k\leq n-1,\]
we obtain
\[\ul\xi_k(a)\leq\ol\xi_k(a)<1,~\forall 0\leq k\leq n-1.\]
On the other hand, we have $\ol{\xi}_n(a)=\ul{\xi}_n(a)=1$
which follows from
\[a_i\sigma_{n-1;i}(a)=\sigma_n(a),~\forall 1\leq i\leq n.\]

Combining \eref{eqn.olxik} and \eref{eqn.sk}, we observe that
\begin{eqnarray*}
\ol\xi_k(a)&=&\frac{a_n\sigma_{k-1;n}(a)}{\sigma_k(a)}
=\frac{a_n\sigma_{k-1;n}(a)}{\sigma_{k;n}(a)+a_n\sigma_{k-1;n}(a)}\\
&\leq&\frac{a_n\sigma_{k;n}(a)}{\sigma_{k+1;n}(a)+a_n\sigma_{k;n}(a)}
=\frac{a_n\sigma_{k;n}(a)}{\sigma_{k+1}(a)}
=\ol\xi_{k+1}(a),
\end{eqnarray*}
where we used the inequality
\[\frac{\sigma_{k-1;n}(a)}{\sigma_{k;n}(a)}
\leq\frac{\sigma_{k;n}(a)}{\sigma_{k+1;n}(a)}\]
which is a variation of the famous Newton inequality (see \cite{HLP34})
\[\sigma_{k-1}(\lambda)\sigma_{k+1}(\lambda)
\leq\left(\sigma_{k}(\lambda)\right)^2,
~\forall\lambda\in\R^n.\]
This completes the proof of \eref{eqn.olxiin},
and similarly of \eref{eqn.ulxiin}.
\end{proof}

\subsection{Polynomials $X,Y,Z,\wh X,\wh Y,\wh Z$
and their properties}\label{subsec.XYZ}

The polynomials appeared in \eref{eqn.sle-af}
play crucial roles in this paper.
For convenience, we give their names by setting
\begin{eqnarray}\label{eqn.Xdef}
X(\lambda)&:=&\sum_{0\leq 2j\leq n}(-1)^j\sigma_{2j}(\lambda)\nonumber\\
&=&1-\sigma_2(\lambda)+\sigma_4(\lambda)-...,
\end{eqnarray}
and
\begin{eqnarray}\label{eqn.Ydef}
Y(\lambda)&:=&\sum_{0\leq 2j+1\leq n}(-1)^j\sigma_{2j+1}(\lambda)\nonumber\\
&=&\sigma_1(\lambda)-\sigma_3(\lambda)+\sigma_5(\lambda)-...,
\end{eqnarray}
for any $\lambda\in\R^n$. Write for short also
\begin{equation}\label{eqn.Hdef}
H(\lambda):=\sum_{i=1}^{n}{\arctan\lambda_i},
~\forall\lambda\in\R^n.
\end{equation}
We get the following elementary relations
according to the trigonometry:
\begin{equation}\label{eqn.Y/X=T}
\frac{Y(\lambda)}{X(\lambda)}=\tan{H(\lambda)}
=\frac{\sin{H(\lambda)}}{\cos{H(\lambda)}},
\end{equation}
and
\begin{equation}\label{eqn.X/Y=C}
\frac{X(\lambda)}{Y(\lambda)}=\cot{H(\lambda)}
=\frac{\cos{H(\lambda)}}{\sin{H(\lambda)}},
\end{equation}
provided the denominators are not zero.
Moreover, we see that
\begin{lemma}\label{lem.samesign}
For any $\lambda\in\R^n$,
$\cos{H(\lambda)}$ and $X(\lambda)$
(respectively, $\sin{H(\lambda)}$ and $Y(\lambda)$)
have the same sign. That is
\begin{equation}\label{eqn.cosX}
\cos H(\lambda)\lessgtr0~\Lra~X(\lambda)\lessgtr0,
\end{equation}
and
\begin{equation}\label{eqn.sinY}
\sin H(\lambda)\lessgtr0~\Lra~Y(\lambda)\lessgtr0.
\end{equation}
\end{lemma}

\begin{proof}
By \eref{eqn.Y/X=T}, it is clear that
\eref{eqn.cosX} and \eref{eqn.sinY} are equivalent,
thus we need only to prove the former.
To do this, the following assertions
(which can also be used as sketches of the proof)
will be adequate.
\begin{enumerate}[\quad($1^\circ$)]
\item
$\{\lambda\in\R^n|\cos H(\lambda)\neq0\}$
has exactly $\wt n+1$ connected components,
and in any two adjacent such components
$\cos H(\lambda)$ have different signs.
Here
\[
\wt n:=\wt n(n):=
\begin{cases}
n-1,~&\text{when $n$ is odd},\\
n,~&\text{when $n$ is even},
\end{cases}
\]
which will be used throughout this proof
and will not be mentioned again.

\item
Since $X(\lambda)=0$ $\Lra$ $\cos H(\lambda)=0$,
\emph{$\{\lambda\in\R^n|X(\lambda)\neq0\}$
also has exactly $\wt n+1$ connected components,
as same as those of $\{\lambda\in\R^n|\cos H(\lambda)\neq0\}$.}
Because, for any $\lambda\in\Gamma^+$,
$X(t\lambda)$ is a polynomial with respect to $t$ of $\wt n$ order,
and it has exactly $\wt n$ different real roots,
we know that $\{\lambda\in\R^n|X(\lambda)\neq0\}\cap\Gamma^+\cap(-\Gamma^+)$
also has exactly $\wt n+1$ connected components,
and in any two adjacent such components $X(\lambda)$ have different signs.
Thus, by smoothness of $X(\lambda)$ and $\cos H(\lambda)$,
we conclude that \emph{$X(\lambda)$ have different signs
in any two adjacent components of $\{\lambda\in\R^n|X(\lambda)\neq0\}$}.

\item
$X(0)=1=\cos H(0)$.
\end{enumerate}

To confirm ($1^\circ$), we need only to note that
\[|DH(\lambda)|
=\left|\left(\frac{1}{1+\lambda_1^2},...,
\frac{1}{1+\lambda_n^2}\right)\right|>0,\]
\begin{equation}
\cos H(\lambda)=0~\Lra~H(\lambda)=\frac{\pi}{2}+k\pi, k\in\Z,
\end{equation}
and there are exactly $\wt n$ different $k\in\Z$ such that
\[\frac{\pi}{2}+k\pi\in\left(-\frac{n\pi}{2},\frac{n\pi}{2}\right).\]

Now consider ($2^\circ$).
One can easily check that $X(t\lambda)$
is a polynomial with respect to $t$ of $\wt n$ order.
Thus it has at most $\wt n$ different complex roots.
In particular, when $\lambda\in\Gamma^+$,
$X(t\lambda)$ has exactly $\wt n$ different real roots.
Indeed, since
\[\lim_{t\ra-\infty}H(t\lambda)=-n\pi/2,\quad
\lim_{t\ra+\infty}H(t\lambda)=n\pi/2\]
and
\[\p_t(H(t\lambda))
=\sum_{i=1}^{n}\frac{\lambda_i}{1+(t\lambda_i)^2}
>0,~\forall t\in\R,\]
we see that the line $\{t\lambda\}_{t\in\R}$
intersects with all the level surfaces $L_{\Theta}$
($\Theta\in(-n\pi/2,n\pi/2)$) once and only once.
This implies that there are exactly $\wt n$ different real $t$
such that $\cos H(t\lambda)=0$.
On the other hand, since it follows from \eref{eqn.X/Y=C} that
\[X(\lambda)=0~\Lra~\cos H(\lambda)=0,\]
we thus conclude that the $\wt n$-order polynomial
$X(t\lambda)$ has exactly $\wt n$ different real roots,
which lie right in the $\wt n$ different level sets $L_\Theta$
($\Theta=\pi/2+k\pi\in(-n\pi/2,n\pi/2)$, $k\in\Z$), respectively.

Hence, $\{\lambda\in\R^n|X(\lambda)\neq0\}\cap\Gamma^+\cap(-\Gamma^+)$
has exactly $\wt n+1$ connected components,
and in any two adjacent such components $X(\lambda)$ have different signs.
Therefore, ($2^\circ$) is clear.

Assertion ($3^\circ$) is easy,
thus the proof of this lemma is completed.
\end{proof}

\medskip

Now set
\begin{eqnarray}\label{eqn.hXdef}
\wh X(\lambda)
&:=&\sum_{0\leq 2j\leq n}(-1)^j\cdot(2j)
\cdot\sigma_{2j}(\lambda)\nonumber\\
&=&-2\sigma_2(\lambda)+4\sigma_4(\lambda)-...,
\end{eqnarray}
and
\begin{eqnarray}\label{eqn.hYdef}
\wh Y(\lambda)
&:=&\sum_{0\leq 2j+1\leq n}(-1)^j\cdot(2j+1)
\cdot\sigma_{2j+1}(\lambda)\nonumber\\
&=&\sigma_1(\lambda)-3\sigma_3(\lambda)+5\sigma_5(\lambda)-...,
\end{eqnarray}
for any $\lambda\in\R^n$.
As a corollary of \lref{lem.samesign}, we obtain
\begin{corollary}\label{cor.XhY-YhX>0=>CY-SX>0}
For any $a\in\R^n$, we have
\[X(a)\wh Y(a)-Y(a)\wh X(a)>0
~\Ra~
\cos{H(a)}\;\wh Y(a)-\sin{H(a)}\;\wh X(a)>0.\]
\end{corollary}

\begin{proof}
($1^\circ$)
If $\cos H(a)\neq0$, then it follows from \lref{lem.samesign} that
\[\frac{\cos{H(a)}}{X(a)}>0.\]
Thus, by \eref{eqn.Y/X=T}, we can deduce that
\begin{eqnarray*}
&~&\cos{H(a)}\;\wh Y(a)-\sin{H(a)}\;\wh X(a)\\
&=&\frac{\cos{H(a)}}{X(a)}\left(X(a)\wh Y(a)-\tan{H(a)}\;X(a)\wh X(a)\right)\\
&=&\frac{\cos{H(a)}}{X(a)}\left(X(a)\wh Y(a)-Y(a)\wh X(a)\right)\\
&>&0.
\end{eqnarray*}

($2^\circ$)
If $\cos H(a)=0$, then $\sin H(a)=\pm1\neq0$.
Thus we can employ a procedure similar to the above one.
That is, by using \[\frac{\sin{H(a)}}{Y(a)}>0\]
according to \lref{lem.samesign},
we can now deduce from \eref{eqn.X/Y=C} that
\begin{eqnarray*}
&~&\cos{H(a)}\;\wh Y(a)-\sin{H(a)}\;\wh X(a)\\
&=&\frac{\sin{H(a)}}{Y(a)}\left(\cot{H(a)}\;Y(a)\wh Y(a)-Y(a)\wh X(a)\right)\\
&=&\frac{\sin{H(a)}}{Y(a)}\left(X(a)\wh Y(a)-Y(a)\wh X(a)\right)\\
&>&0.
\end{eqnarray*}
This finishes the proof of the lemma.
\end{proof}

We discussed in \cref{cor.XhY-YhX>0=>CY-SX>0}
the sign of $X\wh Y-Y\wh X$ and its application.
In fact, for $X\wh Y-Y\wh X$,
we have established the following explicit formula,
which shows immediately that
$X(a)\wh Y(a)-Y(a)\wh X(a)>0$ for all $a\in\Gamma^+$.
\begin{lemma}\label{lem.XhY-YhX=SS}
For any $a\in\R^n$, we have
\[
X(a)\wh Y(a)-Y(a)\wh X(a)
=\sum_{k=1}^{n}
{\sum_{1\leq i_1,i_2,...,i_{k}\leq n}
{a_{i_1}^2a_{i_2}^2...a_{i_{k-1}}^2a_{i_k}}}.
\]
\end{lemma}

\begin{proof}
See the Appendix in \sref{sec.appendix}.
(Since it is too lengthy, we attach it in the end of this paper.)
\end{proof}

\medskip

The notations $X,Y,\wh X$ and $\wh Y$ we introduced above
can all be viewed as ``global'' notations.
To understand better the special Lagrangian equation \eref{eqn.sle-af},
we need to introduce the following ``local'' notations
$c_k,\xi_k$ and their combinations.

For any $\Theta\in(-n\pi/2,n\pi/2)$
and any $k=0,1,...,n$,
let $c_k=c_k(\Theta)$ be the coefficients such that
\[\sum_{k=0}^{n}{c_k(\Theta)\sigma_k(\lambda)}
=\cos\Theta\;Y(\lambda)-\sin\Theta\;X(\lambda),
~\forall\lambda\in\R^n;\]
that is, we set
\begin{equation}\label{eqn.ckdef}
c_k:=c_k(\Theta):=
\begin{cases}
c_{2j}(\Theta):=(-1)^{j+1}\sin\Theta, & \mathrm{if}~k=2j,\\
c_{2j+1}(\Theta):=(-1)^{j}\cos\Theta, & \mathrm{if}~k=2j+1.
\end{cases}
\end{equation}

For any $\Theta\in(-n\pi/2,n\pi/2)$,
any $a\in\R^n$ and any $k=0,1,...,n$, we define
\begin{equation}\label{eqn.xikdef}
\xi_k:=\xi_k(a):=\xi_k(\Theta,a):=
\begin{cases}
\ol\xi_k(a), & \mathrm{if}~c_k(\Theta)>0,\\
\ul\xi_k(a), & \mathrm{if}~c_k(\Theta)\leq 0.
\end{cases}
\end{equation}
By definition and \lref{lem.xik}, it is easy to see that
\begin{equation}\label{Xikck<=xikck}
\Xi_k(a,x)c_k(\Theta)\leq\xi_k(\Theta,a)c_k(\Theta),
\end{equation}
and
\begin{equation}\label{xikck>=k/nck}
\xi_k(\Theta,a)c_k(\Theta)\geq\frac{k}{n}c_k(\Theta),
\end{equation}
for all $k=0,1,...,n$.

\medskip

For any $\Theta\in(-n\pi/2,n\pi/2)$
and any $\lambda\in\R^n$, set
\begin{eqnarray}\label{eqn.Zdef}
Z(\lambda)&:=&Z(\Theta,\lambda)
:=\sum_{k=0}^{n}{c_k(\Theta)\sigma_k(\lambda)}\nonumber\\
&:=&\cos\Theta\;Y(\lambda)-\sin\Theta\;X(\lambda)\nonumber\\
&=&\cos\Theta\sum_{0\leq 2j+1\leq n}(-1)^j\sigma_{2j+1}(\lambda)\nonumber\\
&~&-\sin\Theta\sum_{0\leq 2j\leq n}(-1)^j\sigma_{2j}(\lambda),
\end{eqnarray}
and
\begin{eqnarray}\label{eqn.hZdef}
\wh Z(\lambda)&:=&\wh Z(\Theta,\lambda)
:=\sum_{k=0}^{n}{kc_k(\Theta)\sigma_k(\lambda)}\nonumber\\
&:=&\cos\Theta\;\wh Y(\lambda)-\sin\Theta\;\wh X(\lambda)\nonumber\\
&=&\cos\Theta\sum_{0\leq 2j+1\leq n}(-1)^j(2j+1)\sigma_{2j+1}(\lambda)\nonumber\\
&~&-\sin\Theta\sum_{0\leq 2j\leq n}(-1)^j(2j)\sigma_{2j}(\lambda).
\end{eqnarray}
We see, for any $a\in\R^n$, that
\begin{equation}\label{eqn.Z(ta)}
Z(ta)=\sum{c_k\sigma_k(ta)}
=\sum{c_k\sigma_k(a)t^k}
=\cos\Theta\;Y(ta)-\sin\Theta\;X(ta),
\end{equation}
\begin{equation*}
\wh Z(ta)=\sum{kc_k\sigma_k(a)t^k}
=\cos\Theta\;\wh Y(ta)-\sin\Theta\;\wh X(ta),
\end{equation*}
and
\begin{equation*}
\wh Z(ta)=t\frac{d}{dt}(Z(ta)).
\end{equation*}

As a corollary of \lref{lem.XhY-YhX=SS}, we have
\begin{corollary}\label{cor.xikcksk>0}
Let $a\in\Gamma^+\cap L_\Theta$ with $0<\Theta<n\pi/2$.
Then
\[\sum_{k=0}^{n}{kc_k(\Theta)\sigma_k(a)}>0,\]
and
\[\sum_{k=0}^{n}{\xi_k(\Theta,a)c_k(\Theta)\sigma_k(a)}>0.\]
\end{corollary}

\begin{proof}
By \lref{lem.XhY-YhX=SS}, we have
\[X(a)\wh Y(a)-Y(a)\wh X(a)>0.\]
Since $a\in L_\Theta$ means that $H(a)=\Theta$,
by \eqref{xikck>=k/nck} and \cref{cor.XhY-YhX>0=>CY-SX>0},
we thus deduce that
\[
\sum{\xi_kc_k\sigma_k}
\geq\frac{1}{n}\sum{kc_k\sigma_k}
=\frac{1}{n}\left(\cos\Theta\;\wh Y(a)
-\sin\Theta\;\wh X(a)\right)
>0.
\]
The proof thereby is completed.
\end{proof}

\medskip

For any $\Theta\in(0,n\pi/2)$
and any $a\in L_\Theta\cap\Gamma^+$,
we now define
\begin{equation}\label{eqn.mdef}
m(\Theta,a):=\frac{\sum{kc_k\sigma_k}}{\sum{\xi_kc_k\sigma_k}}
:=\frac{\sum_{k=0}^{n}{kc_k(\Theta)\sigma_k(a)}}
{\sum_{k=0}^{n}{\xi_k(\Theta,a)c_k(\Theta)\sigma_k(a)}}.
\end{equation}
By the proof of \cref{cor.xikcksk>0},
we see immediately that
\begin{equation}\label{eqn.0<m<=n}
0<m(\Theta,a)\leq n,
~\forall a\in L_\Theta\cap\Gamma^+,
~\forall \Theta\in(0,n\pi/2).
\end{equation}
Thus for any $A\in\mathscr{A}_\Theta$,
we have \[2<m(\Theta,a)\leq n.\]

To construct subsolutions of \eref{eqn.sle},
we need $m(\Theta,\lambda(A))>2$.
Though a large class of matrices satisfying this,
there are still some exceptions.
To see this, we now give some examples.

First we know from \eref{eqn.xikrho1=k/n} that
$\xi_k(\tan(\Theta/n)\mathds{1})=k/n$
and therefore
\[m\Big(\Theta,\lambda\big(\tan(\Theta/n)\,I_n\big)\Big)=n>2.\]

Next we consider \eref{eqn.sle} in $\R^3$ or $\R^4$
with the Lagrangian phase $\Theta=\pi$.
Since $A\in\mathscr{A}^0_\pi$ means that
$a=\lambda(A)$ satisfies $H(a)=\pi$,
we have $\sin H(a)=0$.
By \lref{lem.samesign}, we see that
$c_1(a)\sigma_1(a)-c_3(a)\sigma_3(a)=Y(a)=0$,
Thus we obtain
\begin{eqnarray*}
m(\pi,a)
=\frac{c_1\sigma_1-3c_3\sigma_3}
{\xi_1c_1\sigma_1-\xi_3c_3\sigma_3}
=\frac{1-3}
{\ul\xi_1-\ol\xi_3}
>2,
\end{eqnarray*}
which implies that $\mathscr{A}_\pi=\mathscr{A}^0_\pi$
in three or four dimension.

Finally, we give an example in $\R^5$
to show how the distribution of the values of $m(\Theta,a)$
looks like in general.
Let
\[a_\ve:=\left(\tan\left(\frac{\pi}{3}-2\ve\right),
\tan\left(\frac{\pi}{3}-\ve\right),
\tan\left(\frac{\pi}{3}\right),
\tan\left(\frac{\pi}{3}+\ve\right),
\tan\left(\frac{\pi}{3}+2\ve\right)\right),\]
for any $0<\ve<\pi/12$.
Then $a_\ve\in\Gamma^+$
and $H(a_\ve)=5\pi/3\in(3\pi/2,2\pi)$.
Furthermore, we can calculate
(with the help of the computer, for example) that
\[m(\ve):=m(5\pi/3,a_\ve)
=\frac{4\sqrt{3}\cos(4\ve)+4\sqrt{3}\cos(2\ve)+2\sqrt{3}}
{2\sqrt{3}\cos(4\ve)+2\sin(6\ve)+2\sin(2\ve)+3\sin(4\ve)}.\]
Thus $m(\ve)$ is deceasing in $[0,\pi/12]$
with $m(0)=5$, $m(0.2068)\approx2$ and
\[m(0.2618)\approx m(\pi/12)
=\frac{16+4\sqrt{3}}{13}\approx1.7637.\]
This indicates that
in a wide strip around the ray $\set{k\mathds{1}}$,
$m(\Theta,a)$ is larger than $2$ in general;
But when $a$ is close to the boundary of $\Gamma^+$,
$m(\Theta,a)$ appears to be less than $2$.
We would like to mention that we have done a lot of experiments
like the one above, all the results of them give us similar conclusions
even when the Lagrangian phase is subcritical.
For brevity, we will not add them here any more.

\bigskip

Now we study the properties of the polynomial $Z(ta)$,
which will play a crucial role in the construction
of the subsolutions of \eref{eqn.sle}.
First, we have the following basic lemma.
\begin{lemma}\label{lem.S-pN-cN}
For any fixed $a\in L_\Theta\cap\Gamma^+$
with $(n-2)\pi/2\leq\Theta<n\pi/2$,
$Z(ta)$ is a polynomial with respect to $t$ of $N$ order,
and with the leading coefficient $c_N>0$, where
\begin{equation}\label{eqn.Ndef}
N:=N(n,\Theta):=
\begin{cases}
n-1, &\mathrm{when}~\Theta=(n-2)\pi/2,\\
n,   &\mathrm{when}~(n-2)\pi/2<\Theta<n\pi/2,
\end{cases}
\end{equation}
which will be used only in this sense throughout this paper.
\end{lemma}

\begin{proof}
The proof is elementary and hidden in the definition of $c_k$:
\[\begin{cases}
c_{2j}(\Theta):=(-1)^{j+1}\sin\Theta, &\mathrm{if}~k=2j,\\
c_{2j+1}(\Theta):=(-1)^{j}\cos\Theta, &\mathrm{if}~k=2j+1.
\end{cases}\]
For the reader's convenience, we check them case by case.

($1^\circ$)
For $\Theta=(n-2)\pi/2$,
if $n=2j$, then $\Theta=(n-2)\pi/2=(j-1)\pi$ and
\[\begin{cases}
c_n(\Theta)=c_{2j}(\Theta)=(-1)^{j+1}\sin\Theta=0,\\
c_{n-1}(\Theta)=c_{2(j-1)+1}(\Theta)=(-1)^{j-1}\cos\Theta=1>0;
\end{cases}\]
if $n=2j+1$, then $\Theta=(n-2)\pi/2=\pi/2+(j-1)\pi$ and
\[\begin{cases}
c_n(\Theta)=c_{2j+1}(\Theta)=(-1)^{j}\cos\Theta=0,\\
c_{n-1}(\Theta)=c_{2j}(\Theta)=(-1)^{j+1}\sin\Theta=1>0.
\end{cases}\]

($2^\circ$)
For $(n-2)\pi/2<\Theta<n\pi/2$,
if $n=2j$, then $(j-1)\pi<\Theta<j\pi$ and
\[c_n(\Theta)=c_{2j}(\Theta)=(-1)^{j+1}\sin\Theta>0;\]
if $n=2j+1$, then $\pi/2+(j-1)\pi<\Theta<\pi/2+j\pi$ and
\[c_n(\Theta)=c_{2j+1}(\Theta)=(-1)^{j}\cos\Theta>0.\]
This completes the proof of the lemma.
\end{proof}

By definition, it is easy to see that $Z(a)=0$.
This is to say that $t=1$ is a root of $Z(ta)$.
The following lemma asserts that it is indeed a simple root.
\begin{lemma}\label{lem.real-simple-root-1}
For any fixed $a\in\Gamma^+$,
$t=1$ is a simple root of $Z(ta)=0$.
\end{lemma}

\begin{proof}
Invoking the following lemma from algebra:
\begin{enumerate}[\quad]
\item \textsl{(Derivative criterion for simple root)}
\emph{Let $p(x)$ be a polynomial with real coefficients.
Then $x_0$ is a simple root of $p(x)$ if and only if
\[p(x_0)=0\quad\text{and}\quad p'(x_0)\neq0;\]}
\end{enumerate}
since $Z(a)=0$, to show that $t=1$ is a simple root of $Z(ta)=0$,
we now need only to check that
\[\frac{d}{dt}(Z(ta))\Big{|}_{t=1}\neq0.\]
This is evidently true by \cref{cor.xikcksk>0}, since
\[\frac{d}{dt}(Z(ta))\Big{|}_{t=1}=\sum{kc_k\sigma_k(a)}.\]
\end{proof}

In fact, we can prove easily the following lemma.
\begin{lemma}\label{lem.real-simple-root}
For any fixed $a\in L_\Theta\cap\Gamma^+$
with $(n-2)\pi/2\leq\Theta<n\pi/2$,
all the roots of $Z(ta)=0$ are real and simple.
\end{lemma}

\begin{proof}
Because $Z(ta)\triangleq\cos\Theta\;Y(ta)-\sin\Theta\;X(ta)$,
we see that
\[Z(ta)=0
~\Lra~\tan H(ta)=\tan\Theta
~\Lra~H(ta)=\Theta+k\pi.\]
Thus it is clear that
there are exactly $N$ different $t\in\R$ such that
\[-n\pi/2<H(a)=\Theta+k\pi<n\pi/2\]
for some $k\in\Z$,
and that each such $t$ is the root of $Z(ta)=0$.
Since $Z(ta)$ is a polynomial with respect to $t$ of $N$ order,
according to \lref{lem.S-pN-cN},
this shows that
all the roots of $Z(ta)=0$ are real and simple.
\end{proof}

We remark that, using the same method,
\lref{lem.real-simple-root} can be generalized as:
\emph{for any fixed $a\in\Gamma^+$,
all the roots of $Z(ta)=0$ are real and simple.}

\bigskip

The positiveness of $Z(ta)$ and its derivatives,
for each $t\in(1,+\infty)$,
is crucial in constructing the subsolutions of \eref{eqn.sle}.
This is the main reason why we restrict ourselves
in the cases that the Lagrangian phases $\Theta$
are critical and supercritical.
We will now discuss these problems in the following lemma.
\begin{lemma}\label{lem.Z(ta)>0}
Suppose $a\in L_\Theta\cap\Gamma^+$
with $(n-2)\pi/2\leq\Theta<n\pi/2$.
Then
\[\frac{d}{dt}(Z(ta))
=\sum_{k=0}^{n}{kc_k(\Theta)\sigma_k(a)t^{k-1}}
>0,~\forall t\geq1.\]
Furthermore, we have
\[Z(a)=\sum_{k=0}^{n}{c_k(\Theta)\sigma_k(a)}=0,\]
\[Z(ta)=\sum_{k=0}^{n}{c_k(\Theta)\sigma_k(a)t^k}>0,~\forall t>1,\]
and
\[\frac{d^k}{dt^k}(Z(ta))>0,
~\forall t\geq1,~\forall 1\leq k\leq N.\]
\end{lemma}

To prove \lref{lem.Z(ta)>0},
we first introduce the following
simple principle concerning polynomials.
\begin{lemma}\label{lem.P'(x)>0}
Suppose $P(x)$ is a polynomial in $\R$, $P(1)>0$ and
all the roots of $P(x)$ are real and less than $1$.
Then $P'(x)>0$, $\forall x\geq1$.
Furthermore, we have
$\dps\frac{d^k}{dx^k}P(x)>0$,
$\forall k=1,2,...,n$,
$\forall x\geq1$.
\end{lemma}

\begin{proof}
(This simple proof was established with the help of Shanshan Ma.)
Let $a_1,a_2,...,a_n<1$ be the roots of $P(x)$.
Then there exists a constant $K$ such that
\[P(x)=K(x-a_1)(x-a_2)...(x-a_n)=K\prod_{j=1}^{n}(x-a_j).\]
Since $P(1)>0$, we know that $K>0$.
Thus we have
\[P'(x)=K\sum_{i=1}^{n}\prod_{j\neq i}(x-a_j)>0\]
for all $x\geq1$, since $x-a_j\geq1-a_j>0$.

Similarly, we can also obtain
$\dps\frac{d^k}{dx^k}P(x)>0$,
$\forall k=1,2,...,n$,
$\forall x\geq1$.
\end{proof}

We now prove \lref{lem.Z(ta)>0}.
\begin{proof}[\textbf{Proof of \lref{lem.Z(ta)>0}}]
$Z(a)=0$ is already known. We now show that
\[Z(ta)\triangleq\cos\Theta\;Y(ta)-\sin\Theta\;X(ta)>0,~\forall t>1.\]
For any $t>1$, we have $(n-2)\pi/2\leq\Theta=H(a)<H(ta)<n\pi/2$ and
\begin{enumerate}[(A)]
\item
if $\Theta=(n-2)\pi/2$ and
  \begin{enumerate}[(a)]
  \item
  if $n$ is even, then $\sin\Theta=0$, $\cos\Theta\neq0$,
  and $\sin H(ta)$ and $\cos\Theta$ have the same sign.
  Since by \lref{lem.samesign} we know that
  $\sin H(ta)$ and $Y(ta)$ have the same sign,
  we thus obtain $Z(ta)=\cos\Theta\;Y(ta)>0$;

  \item
  if $n$ is odd, then $\cos\Theta=0$, $\sin\Theta\neq0$,
  and $\cos H(ta)$ and $-\sin\Theta$ have the same sign.
  Since by \lref{lem.samesign} we know that
  $\cos H(ta)$ and $X(ta)$ have the same sign,
  we thus get $Z(ta)=-\sin\Theta\;X(ta)>0$;
  \end{enumerate}

\item
if $(n-2)\pi/2<\Theta<n\pi/2$ and
  \begin{enumerate}[(a)]
  \item
  if $n$ is even, then $\sin\Theta\neq0$, $\sin H(ta)\neq0$,
  $\sin\Theta$ and $\sin H(ta)$ have the same sign,
  and $-\cot\theta$ is strictly increasing in $\big((n-2)\pi/2,n\pi/2\big)$.
  Since, by \lref{lem.samesign},
  $\sin H(ta)$ and $Y(ta)$ have the same sign,
  we thus know that $Y(ta)\neq0$ and has the same sign with $\sin\Theta$.
  Hence we have
  \[-\frac{X(ta)}{Y(ta)}=-\cot H(ta)>-\cot H(a)=-\cot\Theta
  =-\frac{\cos\Theta}{\sin\Theta},\]
  and therefore $Z(ta)=\cos\Theta\;Y(ta)-\sin\Theta\;X(ta)>0$;

  \item
  if $n$ is odd, then $\cos\Theta\neq0$, $\cos H(ta)\neq0$,
  $\cos\Theta$ and $\cos H(ta)$ have the same sign,
  and $\tan\theta$ is strictly increasing in $\big((n-2)\pi/2,n\pi/2\big)$.
  Since, by \lref{lem.samesign},
  $\cos H(ta)$ and $X(ta)$ have the same sign,
  we thus know that $X(ta)\neq0$ and has the same sign with $\cos\Theta$.
  Hence we have
  \[\frac{Y(ta)}{X(ta)}=\tan H(ta)>\tan H(a)=\tan\Theta
  =\frac{\sin\Theta}{\cos\Theta},\]
  and therefore $Z(ta)=\cos\Theta\;Y(ta)-\sin\Theta\;X(ta)>0$.
  \end{enumerate}
\end{enumerate}
Thus we have proved, in any case, that $Z(ta)>0$, $\forall t>1$.
This is to say that
all the roots of $Z(ta)$ are less than or equal to $1$.
Invoking \lref{lem.real-simple-root}
and by a slightly modified version of \lref{lem.P'(x)>0},
(i.e., we need to consider $1+\epsilon$ for any small $\epsilon>0$
rather than $1$ in the deducing of
$\frac{d}{dt}(Z(ta))>0$, $\forall t>1$;
as for $\frac{d}{dt}(Z(ta))|_{t=1}>0$,
it has already been proved in \cref{cor.xikcksk>0},
or actually can be viewed
as a corollary of \lref{lem.real-simple-root-1}
or \lref{lem.real-simple-root}.
Once $\frac{d}{dt}(Z(ta))>0$ ($\forall t\geq1$)
is established, the higher order derivative inequalities
can be obtained directly by \lref{lem.P'(x)>0}.)
the rest assertions of this lemma are clear.
\end{proof}
\begin{remark}
To see $Z(ta)>0$ ($\forall t>1$) quickly,
there is an easy point of view which will be stated as below,
although we still believe that the above proof are of more value.
By \lref{lem.S-pN-cN} and the proof of \lref{lem.real-simple-root},
we see that $t$ is a root of the $N$-order polynomial $Z(ta)$
if and only if $H(ta)=\Theta+k\pi$.
Since we now take $\Theta\in\big[(n-2)\pi/2,n\pi/2\big)$,
there are exactly $N$ different such roots $t$
and $t=1$ is the largest one. By \lref{lem.S-pN-cN},
we know that the leading coefficient $c_N$ of $Z(ta)$
is positive, thus $Z(ta)>0$ ($\forall t>1$) is evident.
\end{remark}

The following corollary of \lref{lem.Z(ta)>0} can be viewed
as an extension of \cref{cor.xikcksk>0}.
\begin{corollary}\label{cor.xikcksktk>0}
Suppose $a\in L_\Theta\cap\Gamma^+$
with $(n-2)\pi/2\leq\Theta<n\pi/2$.
Then
\[\sum_{k=0}^{n}{\xi_k(\Theta,a)c_k(\Theta)\sigma_k(a)t^{k-1}}>0,
~\forall t\geq1.\]
\end{corollary}

\begin{proof}
By \eqref{xikck>=k/nck} and \lref{lem.Z(ta)>0}, it is clear that
\[\sum{\xi_kc_k\sigma_kt^{k-1}}
\geq\frac{1}{n}\sum{kc_k\sigma_kt^{k-1}}>0.\]
\end{proof}

\section{Proof of the main theorem}\label{sec.pmaint}

\subsection{Construction of the subsolutions}\label{subsec.csubsol}

In this subsection, we prove the following key lemma
and then use it to construct subsolutions
of the special Lagrangian equation \eref{eqn.sle}.
Note that for the generalized radially symmetric subsolution
$\Phi(x)=\phi(r)$ that we want to construct,
the solution $\psi(r)$ discussed in the following lemma
actually is corresponding to $\phi'(r)/r$,
informally those major eigenvalues of the Hessian $D^2\Phi$
(see the proof of the \lref{lem.Phi-subsol}).
\begin{lemma}\label{lem.psi}
Assume $A\in\mathscr{A}_\Theta$
with $(n-2)\pi/2\leq\Theta<n\pi/2$ and $n\geq3$.
Suppose $a:=(a_1,a_2,...,a_n):=\lambda(A)$,
$0<a_1\leq a_2\leq...\leq a_n$ and $\beta\geq 1$.
Then the problem
\begin{equation}\label{eqn.psi}
\left\{
\begin{aligned}
&r\psi'(r)\sum_{k=0}^{n}
{\xi_k(\Theta,a)c_k(\Theta)\sigma_k(a)(\psi(r))^{k-1}}\\
&\qquad+\sum_{k=0}^{n}{c_k(\Theta)\sigma_k(a)(\psi(r))^k}
=0,~r>1,\\[0.2cm]
&\psi(1)=\beta,
\end{aligned}
\right.
\end{equation}
has a unique smooth solution $\psi(r)=\psi(r,\beta)$ on $[1,+\infty)$,
which satisfies
\begin{enumerate}[\quad(i)]
\item[(i)]
$1\leq\psi(r,\beta)\leq\beta$, $\p_r\psi(r,\beta)\leq0$,
$\forall r\geq1$, $\forall\beta\geq1$.
More specifically,
$\psi(r,1)\equiv 1$, $\psi(1,\beta)\equiv\beta$;
and $1<\psi(r,\beta)<\beta$, $\forall r>1$, $\forall\beta>1$.

\item[(ii)]
$\psi(r)\ra1$ $(r\ra+\infty)$
and $r\psi'(r)\ra0$ $(r\ra+\infty)$.
Furthermore, $\psi(r,\beta)=1+O(r^{-m})~(r\ra+\infty)$,
where $m=m(\Theta,a)\in(2,n]$
and the $O(\cdot)$ depends only on
$n$, $\Theta$, $\lambda(A)$ and $\beta$.

\item[(iii)]
$\psi(r,\beta)$ is continuous
and strictly increasing with respect to $\beta$
and \[\lim_{\beta\ra+\infty}\psi(r,\beta)=+\infty,~\forall r\geq 1.\]
\end{enumerate}
\end{lemma}

\begin{proof}
For simplicity of notation, we will often write
$\psi(r)$ or $\psi(r,\beta)$
(respectively, $c_k(\Theta)$, $\xi_k(\Theta,a)$, $\sigma_k(a)$)
simply as $\psi$ (respectively, $c_k$, $\xi_k$, $\sigma_k$),
when there is no confusion.
The proof of this lemma now will be divided into three steps.

\medskip

\emph{Step 1.}\quad
In this step, we largely follow the strategy
of the one in our previous paper \cite{LL16},
although the foundations of the statements are totally different.

We deduce from \eref{eqn.psi} that
\begin{equation}\label{eqn.psi.d}
\frac{d\psi}{dr}
=-\frac{1}{r}\cdot\frac{\sum_{k=0}^{n}{c_k(\Theta)\sigma_k(a)\psi^k}}
{\sum_{k=0}^{n}{\xi_k(\Theta,a)c_k(\Theta)\sigma_k(a)\psi^{k-1}}}
=:\frac{g(\psi)}{r},
\end{equation}
where we set
\[g(\nu):=-\frac{\sum_{k=0}^{n}{c_k(\Theta)\sigma_k(a)\nu^k}}
{\sum_{k=0}^{n}{\xi_k(\Theta,a)c_k(\Theta)\sigma_k(a)\nu^{k-1}}}.\]
Hence the problem \eref{eqn.psi}
is equivalent to the following problem
\begin{equation}\label{eqn.psi.g}
\left\{
\begin{aligned}
\psi'(r)&=\frac{g(\psi(r))}{r},~r>1,\\
\psi(1)&=\beta.
\end{aligned}
\right.
\end{equation}

If $\beta=1$, then $\psi(r)\equiv 1$ is a solution
of the problem \eref{eqn.psi.g},
since $g(1)=0$ according to \lref{lem.Z(ta)>0}.
Thus, by the uniqueness theorem
for the solution of the ordinary differential equation,
we know that $\psi(r,1)\equiv 1$ is the unique solution
satisfies the problem \eref{eqn.psi.g}.

Now if $\beta>1$, since
\[h(r,\nu):=\frac{g(\nu)}{r}
\in C^{\infty}((1,+\infty)\times(\nu_0,+\infty)),\]
where $0<\nu_0<1$
(note that $\nu_0$ exists, since we have
\[\sum_{k=0}^{n}{\xi_k(\Theta,a)c_k(\Theta)\sigma_k(a)\nu^{k-1}}>0,
~\forall \nu\geq1,\]
according to \cref{cor.xikcksktk>0}),
by the existence theorem
(i.e., the Picard-Lindel\"{o}f theorem)
and the theorem of the maximal interval of existence
for the solution
of the initial value problem
of the ordinary differential equation,
we know that the problem \eref{eqn.psi.g}
has a unique smooth solution $\psi(r)=\psi(r,\beta)$
locally around the initial point
and can be extended to a maximal interval
$[1,\zeta)$, in which $\zeta$
can only be one of the following cases:
\begin{enumerate}[\qquad($1^\circ$)]
\item $\zeta=+\infty$;

\item $\zeta<+\infty$, $\psi(r)$ is unbounded on $[1,\zeta)$;

\item $\zeta<+\infty$, $(r,\psi(r))$ converges to some point on $\{\nu=\nu_0\}$
as $r\ra\zeta-$.
\end{enumerate}

By \lref{lem.Z(ta)>0} and \cref{cor.xikcksktk>0}, we see that
\[\frac{g(\psi(r))}{r}<0,~\forall \psi(r)>1.\]
Thus $\psi(r)=\psi(r,\beta)$
is strictly decreasing with respect to $r$,
this excludes the case ($2^\circ$) above.

We assert that the case ($3^\circ$) can also be excluded.
Otherwise, the solution curve
must intersect with $\{\nu=1\}$ at some point $(r_0,\psi(r_0))$ on it
and then tends to $\{\nu=\nu_0\}$ after crossing it.
But $\psi(r)\equiv 1$ is also a solution through $(r_0,\psi(r_0))$
which contradicts the uniqueness theorem
for the solution of the initial value problem
of the ordinary differential equation.

Thus we complete the proof of the existence and uniqueness
of the solution $\psi(r)=\psi(r,\beta)$
of the problem \eref{eqn.psi} on $[1,+\infty)$.

\medskip

According to the same reason, that is,
$\psi(r,\beta)$ is strictly decreasing with respect to $r$
and the solution curve can not cross $\{\nu=1\}$ provided $\beta>1$,
we see easily that $1<\psi(r,\beta)<\beta$, $\forall r>1$, $\forall\beta>1$.
Thus, assertion $\emph{(i)}$ of the lemma now is clear.

\medskip

\emph{Step 2.}\quad
In view of \lref{lem.real-simple-root} and \lref{lem.Z(ta)>0},
all the roots of the polynomial
$\sum_{k=0}^{n}{c_k\sigma_k\psi^{k-1}}$
are real (less than or equal to one) and simple.
Suppose they are $1>\psi_2>...>\psi_N$,
where
\[
N=N(n,\Theta)\triangleq
\begin{cases}
n-1, &\text{when $\Theta=(n-2)\pi/2$},\\
n,   &\text{when $(n-2)\pi/2<\Theta<n\pi/2$},
\end{cases}
\]
which has been defined in \eref{eqn.Ndef} in \lref{lem.S-pN-cN}.
We then have
\[\sum_{k=0}^{n}{c_k\sigma_k\psi^{k-1}}
=K(\psi-1)(\psi-\psi_2)...(\psi-\psi_N)\]
for some $K>0$ (according to \lref{lem.Z(ta)>0}), and
\begin{equation}\label{eqn.frac-K1=1/m}
\frac{\sum_{k=0}^{n}{\xi_kc_k\sigma_k\psi^{k-1}}}
{\sum_{k=0}^{n}{c_k\sigma_k\psi^k}}
=\frac{K_1}{\psi-1}
+\frac{K_2}{\psi-\psi_2}+...
+\frac{K_N}{\psi-\psi_N}
\end{equation}
for some $K_1,K_2,...,K_N\in\R$.
It is easy to show that
$K_1=1/m$, where
\[
m=m(\Theta,a)\triangleq\frac{\sum_{k=0}^{n}{kc_k\sigma_k}}
{\sum_{k=0}^{n}{\xi_kc_k\sigma_k}}
=-g'(1),
\]
which has been defined in \eref{eqn.mdef} in \ssref{subsec.XYZ}
(for $g'(1)$ see also \eref{eqn.g'nu}).
Indeed, if there are polynomials $P,Q,R$ on $\R$
and $K\in\R$ such that $Q(1)\neq0$ and
\[\frac{P(x)}{(x-1)Q(x)}=\frac{K}{x-1}+\frac{R(x)}{Q(x)},\]
we then have
\[P(x)=KQ(x)+(x-1)R(x).\]
Letting $x=1$, we conclude that
\[K=\frac{P(1)}{Q(1)}
=\frac{P(1)}{\frac{d}{dx}\big((x-1)Q(x)\big)\big{|}_{x=1}},\]
which gives formula to calculate $K_1$.

\medskip

Combining \eref{eqn.psi.d} and \eref{eqn.frac-K1=1/m}, we have
\begin{eqnarray*}
-d\ln r&=&-\frac{dr}{r}
=\frac{\sum_{k=0}^{n}{\xi_kc_k\sigma_k\psi^{k-1}}}
{\sum_{k=0}^{n}{c_k\sigma_k\psi^k}}d\psi\\
&=&\left(\frac{K_1}{\psi-1}
+\frac{K_2}{\psi-\psi_2}+...
+\frac{K_N}{\psi-\psi_N}\right)d\psi.
\end{eqnarray*}
Since $K_1=1/m$, we deduce that
\begin{eqnarray}\label{eqn.mdlnr}
d\ln(r^{-m})&=&-md\ln r
=\left(\frac{1}{\psi-1}+\frac{mK_2}{\psi-\psi_2}+...
+\frac{mK_N}{\psi-\psi_N}\right)d\psi\nonumber\\
&=&d\ln\left((\psi-1)(\psi-\psi_2)^{mK_2}...(\psi-\psi_N)^{mK_N}\right).
\end{eqnarray}
Integrating it from $1$ to $r$
and recalling $\psi(1)=\beta\geq1$,
we obtain
\begin{eqnarray*}
&~&\ln\left((\psi(r)-1)(\psi(r)-\psi_2)^{mK_2}...(\psi(r)-\psi_N)^{mK_N}\right)\\
&=&\ln\left((\beta-1)(\beta-\psi_2)^{mK_2}...(\beta-\psi_N)^{mK_N}\right)
+\ln(r^{-m}),
\end{eqnarray*}
and hence
\begin{eqnarray*}
&~&(\psi(r)-1)(\psi(r)-\psi_2)^{mK_2}...(\psi(r)-\psi_N)^{mK_N}\\
&=&(\beta-1)(\beta-\psi_2)^{mK_2}...(\beta-\psi_N)^{mK_N}r^{-m}\\
&=:&(\beta-1)B(\beta)r^{-m},
\end{eqnarray*}
where we set
\[B(\nu):=(\nu-\psi_2)^{mK_2}...(\nu-\psi_N)^{mK_N}.\]
Thus
\begin{equation}\label{eqn.psimoar}
\frac{\psi(r,\beta)-1}{r^{-m}}
=\frac{(\beta-1)B(\beta)}{B(\psi(r,\beta))}.
\end{equation}
Since $\psi_N<\psi_{N-1}<...<\psi_2<1$
and $1\leq\psi(r,\beta)\leq\beta$, $\forall r\geq 1$,
we have
\begin{equation*}
0\leq\frac{\psi(r,\beta)-1}{r^{-m}}
\leq C(n,\Theta,\lambda(A),\beta),
~\forall r\geq 1.
\end{equation*}
Thus we get
\[\psi(r,\beta)\ra 1~(r\ra+\infty),
~\forall \beta\geq 1.\]
Substituting it into \eref{eqn.psimoar}, we deduce that
\[\frac{\psi(r,\beta)-1}{r^{-m}}
\ra\frac{(\beta-1)B(\beta)}{B(1)}~(r\ra+\infty),
~\forall \beta\geq 1.\]
Therefore
\[\psi(r,\beta)=1+\frac{(\beta-1)B(\beta)}{B(1)}r^{-m}+o(r^{-m})
=1+O(r^{-m})~(r\ra+\infty),\]
where $o(\cdot)$ and $O(\cdot)$
depend only on $n$, $\Theta$, $\lambda(A)$ and $\beta$.
Note also that $r\psi'(r)=g(\psi(r))\ra g(1)=0$ $(r\ra+\infty)$.
Thus the assertion \emph{(ii)} of the lemma is proved.

\medskip

\emph{Step 3.}\quad
By the theorem of the differentiability of the solution
with respect to the initial value,
we can differentiate $\psi(r,\beta)$
with respect to $\beta$ as below:
\[\left\{
\begin{aligned}
&\frac{\p\psi(r,\beta)}{\p r}=\frac{g(\psi(r,\beta))}{r},&\\
&\psi(1,\beta)=\beta;&
\end{aligned}\right.\]
\[\Ra\left\{
\begin{aligned}
&\frac{\p^2 \psi(r,\beta)}{\p \beta\p r}
=\frac{g'(\psi(r,\beta))}{r}\cdot\frac{\p\psi(r,\beta)}{\p\beta},&\\
&\frac{\p\psi(1,\beta)}{\p\beta}=1.&
\end{aligned}\right.\]
Let
\[v(r):=\frac{\p\psi(r,\beta)}{\p\beta}.\]
We have
\[\left\{
\begin{aligned}
&\frac{dv}{dr}
=\frac{g'(\psi(r,\beta))}{r}\cdot v,&\\
&v(1)=1.&
\end{aligned}\right.\]
Therefore we can deduce that
\[\frac{dv}{v}
=\frac{g'(\psi(r,\beta))}{r}dr,\]
and hence
\[\frac{\p\psi(r,\beta)}{\p\beta}=v(r)
=\exp\int_1^r{\frac{g'(\psi(\tau,\beta))}{\tau}d\tau}.\]

By calculation, we have
\begin{eqnarray}\label{eqn.g'nu}
g'(\nu)&=&-\frac{1}
{\left(\sum_{k=0}^{n}{\xi_kc_k\sigma_k\nu^{k-1}}\right)^2}
\Bigg(
\sum_{k=0}^{n}{kc_k\sigma_k\nu^{k-1}}
\sum_{k=0}^{n}{\xi_kc_k\sigma_k\nu^{k-1}}\nonumber\\
&&\qquad\qquad
-\sum_{k=0}^{n}{c_k\sigma_k\nu^{k}}
\sum_{k=0}^{n}{(k-1)\xi_kc_k\sigma_k\nu^{k-2}}
\Bigg).
\end{eqnarray}
Since
$\sum_{k=0}^{n}{\xi_kc_k\sigma_k\nu^{k-1}}>0$ ($\forall\nu\geq1$),
$1\leq\psi(r,\beta)\leq\beta$ ($\forall r\geq1$, $\forall\beta\geq1$),
$g'(1)=-m\in[-n,-2)$ and
\[\lim_{\nu\ra+\infty}g'(\nu)=-\frac{1}{\xi_N}
\in\left[-\frac{n}{n-1},-1\right],\]
we obtain
\[\left|g'(\psi(r,\beta))\right|
\leq C(n,\Theta,\lambda(A))<+\infty,
~\forall r\geq1,\]
and hence
\[0<\frac{\p\psi(r,\beta)}{\p\beta}
\leq r^{C},~\forall r\geq 1.\]
In particular, we see that
$\psi(r,\beta)$ is strictly increasing with respect to $\beta$.

\medskip

Now we come to prove
\[\lim_{\beta\ra+\infty}\psi(r,\beta)=+\infty,
~\forall r\geq 1,\]
by contradiction.
Suppose not. There would exist $r_0\geq1$, $M>1$
and $\{\beta_k\}_{k=1}^{\infty}$, $1<\beta_k\ra+\infty$ $(k\ra+\infty)$
such that $\psi(r_0,\beta_k)\leq M$, $\forall k\in\mathds{Z}^+$.
Note that there are infinitely many $\beta_k>M$ satisfying
$1\leq\psi(r_0,\beta_k)\leq M<\beta_k$.
Since
\[\frac{d\psi}{dr}=\frac{g(\psi)}{r},\]
where
\[g(\nu)=-\frac{\sum{c_k\sigma_k\nu^k}}
{\sum{\xi_kc_k\sigma_k\nu^{k-1}}}\]
satisfies $g(1)=0$
and $0<-g(\nu)<C\nu$ ($\forall \nu>1$)
with $C=C(n,\Theta,\lambda(A))>0$,
we have
\[\frac{d\psi}{C\psi}\leq\frac{d\psi}{-g(\psi)}=-\frac{dr}{r}.\]
Integrating it from $M$ to $\beta_k$
and recalling $\psi(1,\beta_k)=\beta_k$, we get
\[\int_{M}^{\beta_k}\frac{d\psi}{C\psi}
\leq\int_{M}^{\beta_k}\frac{d\psi}{-g(\psi)}
\leq\int_{\psi(r_0,\beta_k)}^{\beta_k}\frac{d\psi}{-g(\psi)}
=-\int_{r_0}^{1}\frac{dr}{r}=\ln r_0<+\infty.\]
Let $\beta_k\ra+\infty$, we have
\[\int_{M}^{\beta_k}\frac{d\psi}{C\psi}\ra+\infty,\]
which is a contradiction.
Hence the assertion \emph{(iii)} of the lemma is proved,
and we thus complete the proof of the whole lemma.
\end{proof}

\medskip

Set
\[\mu_R(\beta):=\int_R^{+\infty}\tau\big(\psi(\tau,\beta)-1\big)d\tau,
\quad\forall R\geq 1,~\forall\beta\geq 1.\]
Note that the integral on the right hand side
is convergent, in view of \lref{lem.psi}-\textsl{(ii)}.
Moreover, as an application of \lref{lem.psi},
we have the following corollary.
\begin{corollary}\label{cor.mub}
$\mu_R(\beta)$ is nonnegative, continuous and
strictly increasing with respect to $\beta$.
Furthermore,
\[\mu_R(\beta)=O(R^{-m+2})~(R\ra+\infty),~\forall\beta\geq 1;\]
and
\begin{equation}\label{eqn.mu-infty}
\mu_R(\beta)\ra+\infty~(\beta\ra+\infty),~\forall R\geq 1.
\end{equation}

\end{corollary}

\begin{proof}
The proof is straightforward 
in light of \lref{lem.psi}-\textsl{(ii),(iii)}.
\end{proof}

\medskip

For any $\alpha,\beta,\gamma\in\R$, $\beta,\gamma\geq1$
and for any \textsl{diagonal} matrix $A\in\mathscr{A}_\Theta^0$, let
\begin{equation}\label{eqn.phidef}
\phi(r):=\phi_{\alpha,\beta,\gamma}(r)
:=\alpha+\int_\gamma^r\tau\psi(\tau,\beta)d\tau,
~\forall r\geq\gamma,
\end{equation}
and let
\begin{equation}\label{eqn.Phidef}
\Phi(x):=\Phi_{\alpha,\beta,\gamma,A}(x):=\phi(r)
:=\phi_{\alpha,\beta,\gamma}(r_A(x)),
~\forall x\in\R^n\setminus{E_\gamma},
\end{equation}
where $r=r_A(x)=\sqrt{x^TAx}$.
Then
\begin{eqnarray}
\phi_{\alpha,\beta,\gamma}(r)
&=&\nonumber\int_\gamma^r\tau\big(\psi(\tau,\beta)-1\big)d\tau
+\frac{1}{2}r^2-\frac{1}{2}\gamma^2+\alpha\\
&=&\frac{1}{2}r^2+\left(\mu_\gamma(\beta)+\alpha
-\frac{1}{2}\gamma^2\right)-\mu_r(\beta)\label{eqn.phi-mu}\\
&=&\frac{1}{2}r^2+\left(\mu_\gamma(\beta)+\alpha
-\frac{1}{2}\gamma^2\right)+O(r^{-m+2})
~(r\ra+\infty),\quad\label{eqn.phi-O}
\end{eqnarray}
according to \cref{cor.mub}.
Furthermore, we have
\begin{lemma}\label{lem.Phi-subsol}
$\Phi$ is a smooth subsolution of \eref{eqn.sle}
in $\R^n\setminus\ol{E_\gamma}$, that is,
\begin{equation}\label{eqn.Phi-subsol}
\sum_{i=1}^{n}\arctan\lambda_i\left(D^2\Phi(x)\right)
\geq\Theta,~\forall x\in\R^n\setminus\ol{E_\gamma}.
\end{equation}
\end{lemma}

\begin{proof}
It is clear that $\phi'(r)=r\psi(r)$
and $\phi''(r)=\psi(r)+r\psi'(r)$.
Since
\[r^2=x^TAx=\sum_{i=1}^{n}a_ix_i^2,\]
we have
\[2r\p_{x_i}r=\p_{x_i}\left(r^2\right)=2a_ix_i
\quad\mathrm{and}\quad
\p_{x_i}r=\frac{a_ix_i}{r}.\]
Thus
\[\p_{x_i}\Phi(x)=\phi'(r)\p_{x_i}r
=\frac{\phi'(r)}{r}a_ix_i,\]
and
\begin{eqnarray*}
\p_{x_ix_j}\Phi(x)
&=&\frac{\phi'(r)}{r}a_i\delta_{ij}+
\frac{\phi''(r)-\frac{\phi'(r)}{r}}{r^2}(a_ix_i)(a_jx_j)\\
&=&\psi(r)a_i\delta_{ij}+\frac{\psi'(r)}{r}(a_ix_i)(a_jx_j).
\end{eqnarray*}
Therefore
\begin{equation}\label{eqn.D2Phi}
D^2\Phi(x)=\left(\psi(r)a_i\delta_{ij}
+\frac{\psi'(r)}{r}(a_ix_i)(a_jx_j)\right)_{n\times n}.
\end{equation}
Applying the formula \lref{lem.skm}, we compute
\begin{eqnarray*}
\sigma_k(\lambda(D^2\Phi))(x)
&=&\sigma_k(a)\psi(r)^k+\frac{\psi'(r)}{r}\psi(r)^{k-1}
\sum_{i=1}^n\sigma_{k-1;i}(a)a_i^2x_i^2\\
&=&\sigma_k(a)\psi^k+\Xi_k(a,x)\sigma_k(a)r\psi^{k-1}\psi'.
\end{eqnarray*}
Thus we deduce, for all $x\in\R^n\setminus\ol{E_\gamma}$, that
\begin{eqnarray*}
&~&Z(\lambda(D^2\Phi))(x)\\
&=&\cos\Theta\;Y(\lambda(D^2\Phi))
-\sin\Theta\;X(\lambda(D^2\Phi))\\
&=&\cos\Theta\sum_{0\leq 2j+1\leq n}(-1)^j\sigma_{2j+1}
(\lambda(D^2\Phi))
-\sin\Theta\sum_{0\leq 2j\leq n}(-1)^j\sigma_{2j}
(\lambda(D^2\Phi))\\
&=&\sum_{k=0}^{n}{c_k(\Theta)\sigma_k
(\lambda(D^2\Phi))}\\
&=&\sum_{k=0}^{n}{c_k(\Theta)\left(\sigma_k(a)\psi^k
+\Xi_k(a,x)\sigma_k(a)r\psi^{k-1}\psi'\right)}\\
&=&\sum_{k=0}^{n}{c_k(\Theta)\sigma_k(a)\psi^k}
+\sum_{k=0}^{n}{\Xi_k(a,x)c_k(\Theta)\sigma_k(a)r\psi^{k-1}\psi'}\\
&\geq&\sum_{k=0}^{n}{c_k(\Theta)\sigma_k(a)\psi^k}
+\sum_{k=0}^{n}{\xi_k(\Theta,a)c_k(\Theta)\sigma_k(a)r\psi^{k-1}\psi'}\\
&=&\sum_{k=0}^{n}{c_k(\Theta)\sigma_k(a)\psi^k}
+r\psi'\sum_{k=0}^{n}{\xi_k(\Theta,a)c_k(\Theta)\sigma_k(a)\psi^{k-1}}\\
&=&0,
\end{eqnarray*}
in which we have used the facts that
$\Xi_k(a,x)c_k(\Theta)\leq\xi_k(\Theta,a)c_k(\Theta)$ for all $k\geq0$,
according to \eqref{Xikck<=xikck}, and that
$\psi(r)\geq1>0$ and $\psi'(r)\leq0$ for all $r\geq1$,
according to \lref{lem.psi}-\textsl{(i)}.

Since $\psi(r)\ra1$ $(r\ra+\infty)$
and $r\psi'(r)\ra0$ $(r\ra+\infty)$ by \lref{lem.psi},
in light of \eref{eqn.D2Phi}, we have
\begin{equation}\label{eqn.laD2PhiTTa}
\begin{aligned}
\left.
\begin{aligned}
D^2\Phi(x)&\ra\mathrm{diag}\{a_1,a_2,...,a_n\}~\\
\lambda(D^2\Phi(x))&\ra a\\
H(\lambda(D^2\Phi(x)))&\ra H(a)=\Theta
\end{aligned}
\right\}
\text{~as $r_A(x)\ra+\infty$.}
\end{aligned}
\end{equation}

Note that, for any $\Theta\in[(n-2)\pi/2,n\pi/2)$,
$L_\Theta:=\{\lambda\in\R^n~|H(\lambda)=\Theta\}$
is one of the $N=N(n,\Theta)$ mutually disjoint components of
\[L_\Theta^\ast:=\{\lambda\in\R^n~|Z(\lambda)=\Theta\},\]
and it is also the last one (or say, the outermost one)
in the $\mathds{1}$-direction.
Since $Z(t\mathds{1})$ is a polynomial of $N$ order,
which has exactly $N$ real and simple roots,
and approaches to $+\infty$ as $t\ra+\infty$,
we conclude that
\begin{equation}\label{eqn.Z>0Z<0}
Z(\lambda)
\begin{cases}
>0& \text{if $H(\lambda)>\Theta$},\\
<0& \text{if $\Theta+\pi<H(\lambda)<\Theta$}.
\end{cases}
\end{equation}

Combing \eref{eqn.laD2PhiTTa}, \eref{eqn.Z>0Z<0}
with $Z(\lambda(D^2\Phi))\geq0$ we derived above,
by the continuity of $\Phi$, we obtain
\[\sum_{i=1}^{n}\arctan\lambda_i(D^2\Phi(x))
=H(\lambda(D^2\Phi(x)))
\geq\Theta,~\forall x\in\R^n\setminus\ol{E_\gamma}.\]
This completes the proof of \lref{lem.Phi-subsol}.
\end{proof}

\subsection{\lref{lem.sle} implies \tref{thm.sle-Theta>0}}\label{subsec.lemsle>thmsle}

In this subsection, we introduce the following lemma
which is a special and simple case of \tref{thm.sle-Theta>0}
with the additional condition that
the matrix $A$ is \textsl{diagonal} and the vector $b$ vanishes.
\begin{lemma}\label{lem.sle}
Let $D$ be a bounded strictly convex domain
in $\R^n$, $n\geq 3$, $\p D\in C^2$
and let $\varphi\in C^2(\p D)$.
Then for any given \textsl{diagonal matrix}
$A\in\mathscr{A}_\Theta$
with $(n-2)\pi/2\leq\Theta<n\pi/2$,
there exists a constant ${c_\ast}$
depending only on $n,D,\Theta,A$
and $\norm{\varphi}_{C^2(\p D)}$,
such that for every $c\geq{c_\ast}$,
there exists a unique viscosity solution
$u\in C^0(\R^n\setminus D)$ of
\begin{equation}\label{eqn.sle-ac}
\left\{
\begin{aligned}
\dps \sum_{i=1}^{n}\arctan\lambda_i&\left(D^2u\right)=\Theta
\quad\mathrm{in}~\R^n\setminus\ol{D},\\
\dps &\qquad~u=\varphi\quad\mathrm{on}~\p D,\\[0.2cm]
\dps \limsup_{|x|\ra+\infty}|x|^{m-2}
&\left|u(x)-\left(\frac{1}{2}x^{T}Ax+c\right)\right|<\infty,
\end{aligned}
\right.
\end{equation}
where $m=m(\Theta,\lambda(A))\in(2,n]$.
\end{lemma}

To prove \tref{thm.sle-Theta>0}, it suffices to prove \lref{lem.sle}.
Indeed, suppose that $D,\varphi,A$ and $b$
satisfy the hypothesis of \tref{thm.sle-Theta>0}.
Consider the decomposition $A=Q^T\Lambda Q$,
where $Q$ is an orthogonal matrix and $\Lambda$ is a diagonal matrix
which satisfies $\lambda(\Lambda)=\lambda(A)$.
Let
\[\tilde{x}:=Qx,\quad\wt{D}:=\left\{Qx|x\in D\right\}\]
and
\[\tilde{\varphi}(\tilde{x}):=\varphi(x)-b^Tx
=\varphi(Q^T\tilde{x})-b^TQ^T\tilde{x}.\]
By \lref{lem.sle}, we conclude that
there exists a constant ${c_\ast}$
depending only on $n,\wt{D},\Theta,\Lambda$
and $\norm{\tilde{\varphi}}_{C^2(\p\wt{D})}$,
such that for every $c\geq{c_\ast}$,
there exists a unique viscosity solution
$\tilde{u}\in C^0(\R^n\setminus\wt{D})$ of
\begin{equation}\label{eqn.sle-nc}
\left\{
\begin{aligned}
\dps \sum_{i=1}^{n}\arctan\lambda_i&\left(D^2\tilde{u}\right)=\Theta
\quad\mathrm{in}~\R^n\setminus\ol{\wt{D}},\\
\dps &\qquad~\tilde{u}=\tilde{\varphi}
\quad\mathrm{on}~\p\wt{D},\\[0.2cm]
\dps \limsup_{|\tilde{x}|\ra+\infty}|\tilde{x}|^{m-2}
&\left|\tilde{u}(\tilde{x})-\left(\frac{1}{2}\tilde{x}^{T}\Lambda\tilde{x}
+c\right)\right|<\infty,
\end{aligned}
\right.
\end{equation}
where $m=m(\Theta,\lambda(\Lambda))=m(\Theta,\lambda(A))\in(2,n]$.
Let
\[u(x):=\tilde{u}(\tilde{x})+b^Tx=\tilde{u}(Qx)+b^Tx
=\tilde{u}(\tilde{x})+b^TQ^T\tilde{x}.\]
We assert that $u$ is the solution
of \eref{eqn.sle-abc} in \tref{thm.sle-Theta>0}.
To show this, we need only to note that
\[D^2u(x)=Q^TD^2\tilde{u}(\tilde{x})Q,\quad
\lambda\left(D^2u(x)\right)=\lambda\left(D^2\tilde{u}(\tilde{x})\right);\]
\[u=\varphi\quad\mathrm{on}~\p D;\]
and
\begin{eqnarray*}
&&\dps |\tilde{x}|^{m-2}\left|\tilde{u}(\tilde{x})
-\left(\frac{1}{2}\tilde{x}^{T}\Lambda\tilde{x}+c\right)\right|\\
&=&\dps \left(x^TQ^TQx\right)^{(m-2)/2}\left|u(x)-b^Tx
-\left(\frac{1}{2}x^TQ^T\Lambda Qx+c\right)\right|\\
&=&\dps |x|^{m-2}
\left|u(x)-\left(\frac{1}{2}x^{T}Ax+b^Tx+c\right)\right|.
\end{eqnarray*}
Thus we have proved that
\tref{thm.sle-Theta>0} can be established by \lref{lem.sle}.

\subsection{Proof of \lref{lem.sle}}\label{subsec.pr-lem-sle}

Now we use the Perron's method to prove \lref{lem.sle}.
The procedures employed here are standard
and quite similar to the ones in our previous paper \cite{LL16}.
For the reader's convenience, we give the full details.
\begin{proof}[\textbf{Proof of \lref{lem.sle}}]
We may assume without loss of generality that
$E_1\subset\subset D\subset\subset E_{\bar{r}}
\subset\subset E_{\hat{r}}$
and $a:=(a_1,a_2,...,a_n):=\lambda(A)$ with
$0<a_1\leq a_2\leq...\leq a_n$.
The proof now will be divided into three steps.

\medskip

\emph{Step 1.}\quad
Let
\[\eta:=\inf_{\substack{x\in \ol{E_{\bar{r}}}\setminus D\\
\xi\in\p D}}Q_\xi(x), \quad
Q(x):=\sup_{\xi\in\p D}Q_\xi(x)\]
and
\[\Phi_\beta(x):=\eta+\int_{\bar{r}}^{r_A(x)}\tau\psi(\tau,\beta)d\tau,
\quad\forall r_A(x)\geq 1,~\forall \beta\geq 1,\]
where $Q_{\xi}(x)$ and $\psi(r,\beta)$ are given
by \lref{lem.Qxi} and \lref{lem.psi}, respectively.
Then we have
\begin{enumerate}[\quad(1)]
\item Since $Q$ is the supremum of
a collection of smooth solutions $\{Q_\xi\}$ of \eref{eqn.sle},
it is a continuous subsolution of \eref{eqn.sle}, i.e.,
\[\sum_{i=1}^{n}\arctan\lambda_i\left(D^2Q\right)\geq\Theta\]
in $\R^n\setminus \ol{D}$ in the viscosity sense
(see \cite[Proposition 2.2]{Ish89}).

\item $Q=\varphi$ on $\p D$.
To prove this we need only to show that
for any $\xi\in\p D$, $Q(\xi)=\varphi(\xi)$.
This is obvious since $Q_\xi\leq\varphi$ on $\ol{D}$
and $Q_\xi(\xi)=\varphi(\xi)$,
according to \rref{rmk.Qxi}-\textsl{(1)}.

\item By \lref{lem.Phi-subsol},
$\Phi_\beta$ is a smooth subsolution of \eref{eqn.sle}
in $\R^n\setminus\ol{D}$.

\item $\Phi_\beta\leq\varphi$ on $\p D$
and $\Phi_\beta\leq Q$ on $\ol{E_{\bar{r}}}\setminus D$.
To show them we note that
$\Phi_\beta(x)$ is strictly increasing
with respect to $r_A(x)$ since
$\psi(r,\beta)\geq 1>0$ by \lref{lem.psi}-\textsl{(i)}.
Invoking $\Phi_\beta=\eta$ on $\p E_{\bar{r}}$
and $\eta\leq Q$ on $\ol{E_{\bar{r}}}\setminus D$
by their definitions, we have $\Phi_\beta\leq\eta\leq Q$
on $\ol{E_{\bar{r}}}\setminus D$.
On the other hand,
according to \rref{rmk.Qxi}-\textsl{(1)}, we have
$Q_\xi\leq\varphi$ on $\ol{D}$ which implies that
$\eta\leq\varphi$ on $\ol{D}$.
Combining these two aspects we deduce that
$\Phi_\beta\leq\eta\leq\varphi$ on $\p D$.

\item $\Phi_\beta(x)$ is strictly increasing
with respect to $\beta$ and
\begin{equation}\label{eqn.Phib}
\lim_{\beta\ra+\infty}\Phi_\beta(x)=+\infty,
~\forall r_A(x)\geq 1,
\end{equation}
by the definition of $\Phi_\beta(x)$
and \lref{lem.psi}-\textsl{(iii)}.

\item As showed in \eref{eqn.phi-mu} and \eref{eqn.phi-O},
for any $\beta\geq 1$, we have
\begin{eqnarray*}
\Phi_\beta(x)
&=&\eta+\int_{\bar{r}}^{r_A(x)}\tau \psi(\tau,\beta)d\tau\\
&=&\eta+\frac{1}{2}(r_A(x)^2-\bar{r}^2)
+\int_{\bar{r}}^{r_A(x)}\tau\big(\psi(\tau,\beta)-1\big)d\tau\\
&=&\frac{1}{2}r_A(x)^2+\left(\eta-\frac{1}{2}\bar{r}^2
+\mu_{\bar{r}}(\beta)\right)-\mu_{r_A(x)}(\beta)\\
&=&\frac{1}{2}r_A(x)^2+\mu(\beta)-\mu_{r_A(x)}(\beta)\\
&=&\frac{1}{2}x^TAx+\mu(\beta)+O\left(|x|^{-m+2}\right)
~(|x|\ra+\infty),
\end{eqnarray*}
where we set
\[\mu(\beta):=\eta-\frac{1}{2}\bar{r}^2+\mu_{\bar{r}}(\beta),\]
and used the fact that $x^TAx=O(|x|^2)~(|x|\ra+\infty)$
since $\lambda(A)\in\Gamma^+$.
\end{enumerate}

\medskip

\emph{Step 2.}\quad
For fixed $\hat{r}>\bar{r}$, there exists $\hat{\beta}>1$ such that
\[\min_{\p E_{\hat{r}}}\Phi_{\hat{\beta}}
>\max_{\p E_{\hat{r}}}Q,\]
in light of \eref{eqn.Phib}.
Thus we obtain
\begin{equation}\label{eqn.PhiQ}
\Phi_{\hat{\beta}}>Q\quad\mathrm{on}~\p E_{\hat{r}}.
\end{equation}

Let
\[{c_\ast}:=\max\left\{\eta,\mu(\hat{\beta}),\bar{c}\right\},\]
where the $\bar{c}$ comes from \rref{rmk.Qxi}-\textsl{(3)},
and hereafter fix $c\geq{c_\ast}$.

By \lref{lem.psi} and \cref{cor.mub} we deduce that
\[\psi(r,1)\equiv 1\Ra\mu_{\bar{r}}(1)=0
\Ra\mu(1)=\eta-\frac{1}{2}\bar{r}^2<\eta\leq{c_\ast}\leq c,\]
and
\[\lim_{\beta\ra+\infty}\mu_{\bar{r}}(\beta)=+\infty
\Ra\lim_{\beta\ra+\infty}\mu(\beta)=+\infty.\]
On the other hand,
it follows from \cref{cor.mub}
that $\mu(\beta)$ is continuous and
strictly increasing with respect to $\beta$
\big{(}which indicates that the inverse of $\mu(\beta)$ exists
and $\mu^{-1}$ is strictly increasing\big{)}.
Thus there exists a unique $\beta(c)$ such that $\mu(\beta(c))=c$.
Then we have
\[\Phi_{\beta(c)}(x)=\frac{1}{2}r_A(x)^2+c-\mu_{r_A(x)}(\beta(c))
=\frac{1}{2}x^TAx+c+O\left(|x|^{-m+2}\right)~(|x|\ra+\infty),\]
and
\[\beta(c)=\mu^{-1}(c)\geq\mu^{-1}({c_\ast})\geq\hat{\beta}.\]
Invoking the monotonicity of $\Phi_\beta$
with respect to $\beta$ and \eref{eqn.PhiQ}, we obtain
\begin{equation}\label{eqn.PhigQ}
\Phi_{\beta(c)}\geq\Phi_{\hat{\beta}}>Q\quad\mathrm{on}
~\p E_{\hat{r}}.
\end{equation}
Note that we already know
\[\Phi_{\beta(c)}\leq Q\quad\mathrm{on}
~\ol{E_{\bar{r}}}\setminus D,\]
from (4) of \emph{Step 1}.

\medskip

Let
\begin{equation*}
\ul{u}(x):=
\begin{cases}
\max\left\{\Phi_{\beta(c)}(x),Q(x)\right\},& x\in E_{\hat{r}}\setminus D,\\
\Phi_{\beta(c)}(x),& x\in \R^n\setminus E_{\hat{r}}.
\end{cases}
\end{equation*}
Then we have
\begin{enumerate}[\quad(1)]
\item $\ul{u}$ is continuous and satisfies
\[\frac{\sigma_k(\lambda(D^2{\ul{u}}))}
{\sigma_l(\lambda(D^2{\ul{u}}))}\geq1\]
in $\R^n\setminus\ol{D}$ in the viscosity sense,
by (1) and (3) of \emph{Step 1}.

\item $\ul{u}=Q=\varphi$ on $\p D$,
by (2) of \emph{Step 1}.

\item If $r_A(x)$ is large enough, then
\[\ul{u}(x)=\Phi_{\beta(c)}(x)
=\frac{1}{2}x^TAx+c+O\left(|x|^{-m+2}\right)~(|x|\ra+\infty).\]
\end{enumerate}

\medskip

\emph{Step 3.}\quad
Let
\[\ol{u}(x):=\frac{1}{2}x^TAx+c,
~\forall x\in\R^n.\]
Then $\ol{u}$ is obviously a supersolution and
\[\lim_{|x|\ra+\infty}\left(\ul{u}-\ol{u}\right)(x)=0.\]


To use the Perron's method to establish \lref{lem.sle},
we now need only to prove that
\[\ul{u}\leq \ol{u}\quad\mathrm{in}~\R^n\setminus D.\]
In fact, since
\[\mu_{r_A(x)}(\beta)\geq 0,
\quad\forall x\in\R^n\setminus E_1,
~\forall\beta\geq 1,\]
according to \cref{cor.mub}, we have
\begin{equation}\label{eqn.Phiolu}
\Phi_{\beta(c)}(x)=\frac{1}{2}x^TAx+c-\mu_{r_A(x)}(\beta(c))
\leq\frac{1}{2}x^TAx+c=\ol{u}(x),~\forall x\in\R^n\setminus D.
\end{equation}
On the other hand, for every $\xi\in\p D$, since
\[Q_\xi(x)\leq\frac{1}{2}x^TAx+\bar{c}
\leq\frac{1}{2}x^TAx+{c_\ast}
\leq\frac{1}{2}x^TAx+c=\ol{u}(x),
~\forall x\in\p D,\]
and
\[Q_\xi\leq Q<\Phi_{\beta(c)}\leq\ol{u}
\quad\mathrm{on}~\p E_{\hat{r}}\]
follows from \eref{eqn.PhigQ} and \eref{eqn.Phiolu},
we obtain
\[Q_\xi\leq\ol{u}\quad\mathrm{on}
~\p\left(E_{\hat{r}}\setminus D\right).\]
In view of
\[\sum_{i=1}^{n}\arctan\lambda_i\left(D^2{Q_\xi}\right)
=\Theta=\sum_{i=1}^{n}\arctan\lambda_i\left(D^2{\ol{u}}\right)
\quad\mathrm{in}~E_{\hat{r}}\setminus D,\]
we deduce from the comparison principle that
\[Q_\xi\leq\ol{u}\quad\mathrm{in}~E_{\hat{r}}\setminus D.\]
Hence
\begin{equation}\label{eqn.Qu}
Q\leq\ol{u}\quad\mathrm{in}~E_{\hat{r}}\setminus D.
\end{equation}
Combining \eref{eqn.Phiolu} and \eref{eqn.Qu},
by the definition of $\ul{u}$, we get
\[\ul{u}\leq \ol{u}\quad\mathrm{in}~\R^n\setminus D.\]
This finishes the proof of \lref{lem.sle}.
\end{proof}

\section{Appendix: proof of \lref{lem.XhY-YhX=SS}}\label{sec.appendix}

\begin{lemma}[\lref{lem.XhY-YhX=SS}]\label{lem.A.XhY-YhX=SS}
For any $a\in\R^n$,
\begin{equation}\label{eqn.wZ=SS}
\wh Z_\ast(a):=X(a)\wh Y(a)-Y(a)\wh X(a)
=\sum_{k=1}^{n}
{\sum_{1\leq i_1,i_2,...,i_{k}\leq n}
{a_{i_1}^2a_{i_2}^2...a_{i_{k-1}}^2a_{i_k}}}.
\end{equation}
\end{lemma}

Recall that we have defined
\begin{eqnarray*}
X(\lambda)&:=&\sum_{0\leq 2j\leq n}(-1)^j\sigma_{2j}(\lambda)\\
&=&1-\sigma_2(\lambda)+\sigma_4(\lambda)-...,
\end{eqnarray*}
\begin{eqnarray*}
Y(\lambda)&:=&\sum_{0\leq 2j+1\leq n}(-1)^j\sigma_{2j+1}(\lambda)\\
&=&\sigma_1(\lambda)-\sigma_3(\lambda)+\sigma_5(\lambda)-...,
\end{eqnarray*}
\begin{eqnarray*}
\wh X(\lambda)
&:=&\sum_{0\leq 2j\leq n}(-1)^j\cdot(2j)\cdot\sigma_{2j}(\lambda)\\
&=&-2\sigma_2(\lambda)+4\sigma_4(\lambda)-...
\end{eqnarray*}
and
\begin{eqnarray*}
\wh Y(\lambda)
&:=&\sum_{0\leq 2j+1\leq n}(-1)^j\cdot(2j+1)\cdot\sigma_{2j+1}(\lambda)\\
&=&\sigma_1(\lambda)-3\sigma_3(\lambda)+5\sigma_5(\lambda)-....
\end{eqnarray*}

\begin{proof}[\textbf{Proof of \lref{lem.A.XhY-YhX=SS}}]
The proof will be divided into four steps.

\emph{Step 1.}\quad
We have
\begin{eqnarray}
\wh Z_\ast(a)
&=&\sum_{0\leq 2i\leq n}(-1)^i\sigma_{2i}(a)
\sum_{0\leq 2j+1\leq n}(-1)^j(2j+1)\sigma_{2j+1}(a)\nonumber\\
&&\quad-\sum_{0\leq 2i\leq n}(-1)^i(2i)\sigma_{2i}(a)
\sum_{0\leq 2j+1\leq n}(-1)^j\sigma_{2j+1}(a)\nonumber\\
&=&\sum_{\substack{0\leq 2i\leq n\\0\leq 2j+1\leq n}}
(-1)^{i+j}(2j-2i+1)\sigma_{2i}(a)\sigma_{2j+1}(a)\nonumber\\
&=&\sum_{p=0}^{\wh n}\sum_{q=0}^{p}(-1)^q(2q+1)
\sigma_{p-q}(a)\sigma_{p+q+1}(a)\nonumber\\
&&+\sum_{p=\wh n+1}^{n-1}\sum_{q=0}^{n-1-p}(-1)^q(2q+1)
\sigma_{p-q}(a)\sigma_{p+q+1}(a),\label{eqn.pwhn}
\end{eqnarray}
where
\begin{equation}\label{eqn.whn}
\wh n:=
\begin{cases}
\frac{n-1}{2} &\text{if $n$ is odd},\\
\frac{n-2}{2} &\text{if $n$ is even},
\end{cases}
\end{equation}
and the last equality follows from the rearrangement
(see for example \textsl{Figure} \ref{figure}) realized by
\begin{equation}
\begin{cases}
p=i+j,\\
\dps q=
\begin{cases}
j-i& \text{if $i\leq j$},\\
i-j-1& \text{if $i>j$},
\end{cases}
\end{cases}
\end{equation}
or, equivalently, by
\begin{equation}
\left.
\begin{aligned}
i=\frac{p-q}{2}\\
j=\frac{p+q}{2}
\end{aligned}
\right\}
\text{if $p+q$ is even},
\end{equation}
in the opposite direction.
Note that if $0\leq i\leq j$,
then $p:=i+j\geq0$, $q:=j-i\geq0$, $2j+1>2i$,
$(-1)^{i+j}=(-1)^{j-i}=(-1)^{q}$, $2j-2i+1=2q+1$,
$2i=p-q$ and $2j+1=p+q+1$;
if $0\leq j<i$,
then $p:=i+j\geq0$, $q:=i-j-1\geq0$, $2j+1<2i$,
$(-1)^{i+j}=-(-1)^{i-j-1}=-(-1)^{q}$,
$2j-2i+1=-(2i-2j-1)=-(2q+1)$,
$2i=p+q+1$ and $2j+1=p-q$.

\begin{figure}[!h]
\centering
\includegraphics[width=0.45\textwidth]{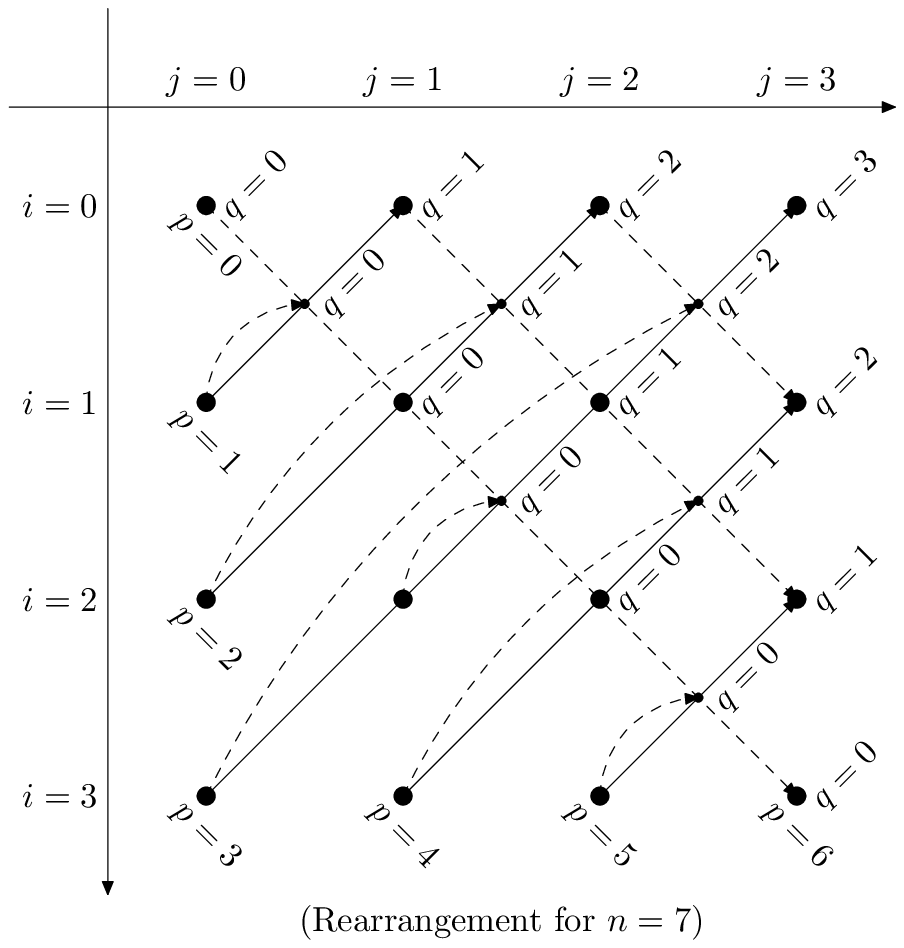}
\includegraphics[width=0.50\textwidth]{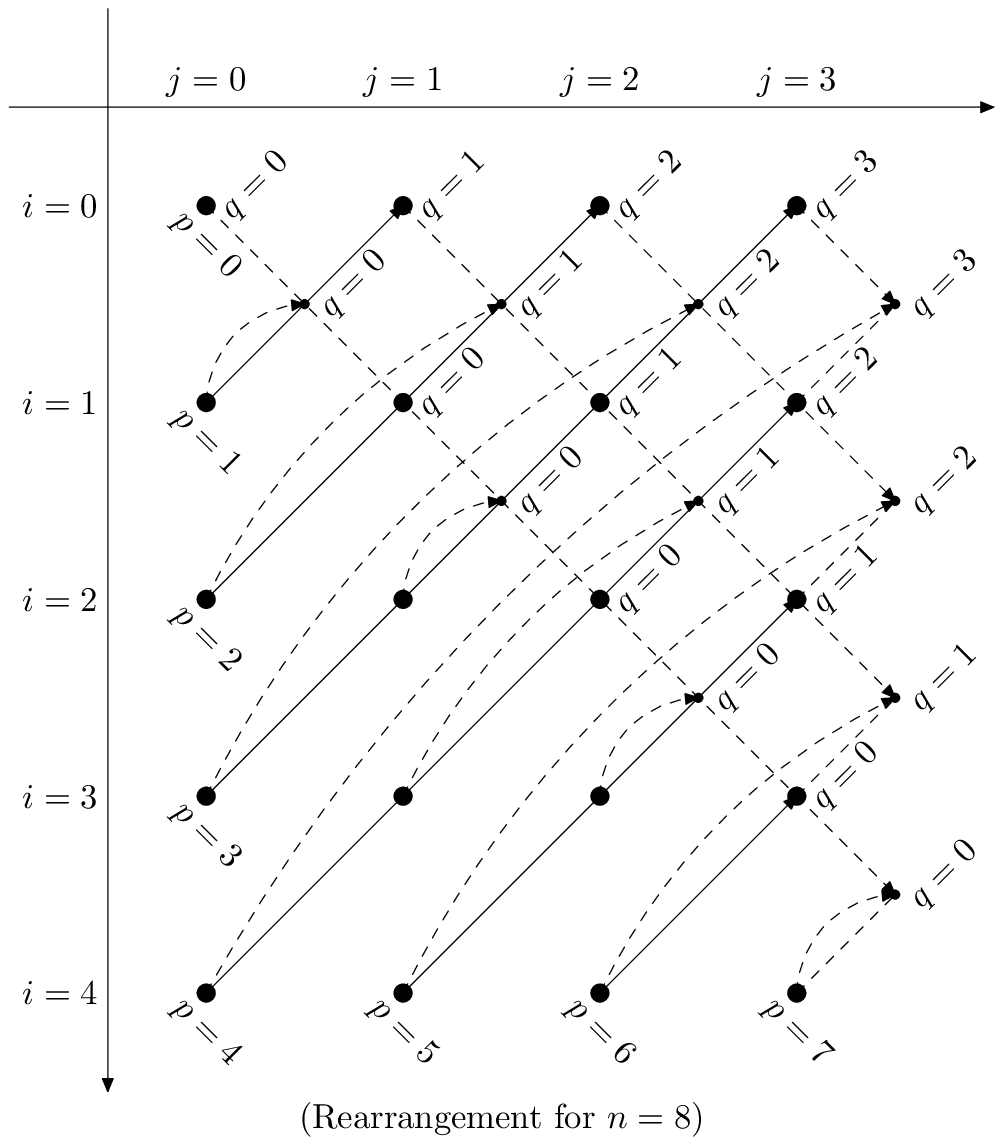}
\caption{Rearrangement for odd and/or even dimension}\label{figure}
\end{figure}

Thus, in order to calculate $\wh Z_\ast(a)$,
we need only to consider
\begin{equation}\label{eqn.sumqgtO}
\sum_{q=0}^{\wh p}(-1)^q(2q+1)
\sigma_{p-q}(a)\sigma_{p+q+1}(a)
\end{equation}
for all $0\leq p\leq n-1$, where
\begin{equation*}
\wh p:=
\begin{cases}
p, &p\leq\frac{n-1}{2},\\
n-1-p, &p>\frac{n-1}{2}.
\end{cases}
\end{equation*}
Note that the $\wh p$ here,
unlike the $\wh n$ in \eref{eqn.whn},
is independent of the parity (oddness or evenness)
of the dimension $n$.

\medskip

The calculation of \eref{eqn.sumqgtO}
will be given in \textsl{Step 4}.
To do this, we need some other preparations
which will be presented in the following two steps.

\medskip

\emph{Step 2.}\quad
We now prove that
\begin{equation}\label{eqn.QIO}
\sum_{q=0}^{Q}{(-1)^q(2q+1)C_{2Q+1}^{Q-q}}=
\begin{cases}
0& \text{if $Q\geq1$},\\
1& \text{if $Q=0$}.
\end{cases}
\end{equation}

For $Q=0$, \eref{eqn.QIO} is obviously true.
So we may assume that $Q\geq1$.
Then
\begin{eqnarray*}
&&\sum_{q=0}^{Q}{(-1)^q(2q+1)C_{2Q+1}^{Q-q}}\\
&=&(2Q+1)\left(\sum_{q=0}^{Q}{(-1)^{q}C_{2Q}^{Q+q}}
+\sum_{q=0}^{Q}{(-1)^{q+1}C_{2Q}^{Q+q+1}}\right)\\
&=&(2Q+1)\left(C_{2Q}^{Q}
+\sum_{q=1}^{Q}{(-1)^{q}C_{2Q}^{Q+q}}
+\sum_{q=0}^{Q-1}{(-1)^{q+1}C_{2Q}^{Q+q+1}}\right)\\
&=&(2Q+1)\left(C_{2Q}^{Q}
+2\sum_{q=1}^{Q}{(-1)^{q}C_{2Q}^{Q+q}}\right)\\
&=&0,
\end{eqnarray*}
where we have used the formulas
\begin{equation}\label{eqn.cqchai}
(2q+1)C_{2Q+1}^{Q-q}
=(2Q+1)\left(C_{2Q}^{Q+q}-C_{2Q}^{Q+q+1}\right),
~\forall 0\leq q\leq Q,
\end{equation}
and
\begin{equation}\label{eqn.cqq}
\sum_{q=1}^{Q}{(-1)^{q}C_{2Q}^{Q+q}}
=-\frac{1}{2}C_{2Q}^{Q},
~\forall Q\geq1.
\end{equation}
In order to complete the proof of \eref{eqn.QIO},
we must prove \eref{eqn.cqchai} and \eref{eqn.cqq}.
To verify assertion \eref{eqn.cqchai}, we check directly that
\begin{eqnarray*}
&&(2Q+1)\left(C_{2Q}^{Q+q}-C_{2Q}^{Q+q+1}\right)\\
&=&(2Q+1)\left(\frac{(2Q)!}{(Q+q)!\,(Q-q)!}
-\frac{(2Q)!}{(Q+q+1)!\,(Q-q-1)!}\right)\\
&=&\frac{(2Q+1)!\,(Q+q+1-(Q-q))}{(Q+q+1)!\,(Q-q)!}\\
&=&\frac{(2Q+1)!\,(2q+1)}{(Q+q+1)!\,(Q-q)!}\\
&=&(2q+1)C_{2Q+1}^{Q-q}.
\end{eqnarray*}
To prove assertion \eref{eqn.cqq},
we first note that this is equivalent to
\begin{equation}\label{eqn.cqqq}
\sum_{q=0}^{Q-1}{(-1)^{q}C_{2Q}^{q}}
=\frac{1}{2}(-1)^{Q+1}C_{2Q}^{Q},
~\forall Q\geq1,
\end{equation}
since
\begin{eqnarray*}
\sum_{q=1}^{Q}{(-1)^{q}C_{2Q}^{Q+q}}
&=&\sum_{q=1}^{Q}{(-1)^{q}C_{2Q}^{Q-q}}
=\sum_{j=0}^{Q-1}{(-1)^{Q-j}C_{2Q}^{j}}\\
&=&(-1)^{Q}\sum_{j=0}^{Q-1}{(-1)^{j}C_{2Q}^{j}}
=(-1)^{Q}\sum_{q=0}^{Q-1}{(-1)^{q}C_{2Q}^{q}}.
\end{eqnarray*}
Next, to show \eref{eqn.cqqq}, we need only to note that
\begin{eqnarray*}
0=(1+x)^{2Q}\big{|}_{x=-1}
&=&\sum_{q=0}^{2Q}{(-1)^{q}C_{2Q}^{q}}\\
&=&\sum_{q=0}^{Q-1}{(-1)^{q}C_{2Q}^{q}}+(-1)^{Q}C_{2Q}^{Q}
+\sum_{q=Q+1}^{2Q}{(-1)^{q}C_{2Q}^{q}}\\
&=&2\sum_{q=0}^{Q-1}{(-1)^{q}C_{2Q}^{q}}+(-1)^{Q}C_{2Q}^{Q},
\end{eqnarray*}
where the last equality holds since
\[
\sum_{q=Q+1}^{2Q}{(-1)^{q}C_{2Q}^{q}}
=\sum_{j=0}^{Q-1}{(-1)^{2Q-j}C_{2Q}^{2Q-j}}
=\sum_{q=0}^{Q-1}{(-1)^{q}C_{2Q}^{q}}.
\]
Thus the proof of \eref{eqn.QIO} is complete.

\emph{Step 3.}\quad
For any $0\leq j\leq k\leq n$ and any $a=(a_1,a_2,...,a_n)\in\R^n$, set
\[S_k^j:=S_k^j(a):=\sum_{1\leq i_1,i_2,...,i_k\leq n}
{a_{i_1}^2a_{i_2}^2...a_{i_j}^2a_{i_{j+1}}...a_{i_k}}.\]
Note that $S_k^0=\sigma_k$
and $S_k^j(\mathbb{\mathds{1}})=C_{n}^{k}C_{k}^{j}$,
where $\mathds{1}:=(1,1,...,1)\in\R^n$,
which means that $S_k^j$ is
a polynomial of $C_{n}^{k}C_{k}^{j}$ terms
and can be viewed as a generalization
of the elementary symmetric polynomial $\sigma_k$.

The main purpose in this step is
to prove the following elementary decompositions
\begin{equation}\label{eqn.sjkl}
\sigma_j(a)\sigma_k(a)=\sum_{h=0}^{j}
{C_{j+k-2h}^{j-h}S_{j+k-h}^{h}(a)},
~\forall 0\leq j\leq k\leq n,~j+k\leq n,
\end{equation}
and
\begin{equation}\label{eqn.sjkg}
\sigma_j(a)\sigma_k(a)=\sum_{h=0}^{n-k}
{C_{2n-j-k-2h}^{n-j-h}S_{n-h}^{j+k-n+h}(a)},
~\forall 0\leq j\leq k\leq n,~j+k\geq n,
\end{equation}
for any $a\in\R^n$.

To do this, we first observe that
$\sigma_j(a)\sigma_k(a)$ is a polynomial
with terms of the $S_L^M(a)$ types,
where $0\leq L\leq M\leq n$ and $L+M=j+k$.
Thus for $j+k\leq n$,
any term of $\sigma_j\sigma_k$ can only be one of the elements of
\[\set{S_{j+k}^{0},S_{j+k-1}^{1},S_{j+k-2}^{2},
...,S_{j+k-h}^{h},...,S_{k+1}^{j-1},S_{k}^{j}},\]
and, for $j+k\geq n$,
any term of $\sigma_j\sigma_k$ can only be one of the elements of
\[\set{S_{n}^{j+k-n},S_{n-1}^{j+k-n+1},S_{n-2}^{j+k-n+2},
...,S_{n-h}^{j+k-n+h},...,S_{k+1}^{j-1},S_{k}^{j}}.\]

Hence it remains to determine the coefficient of each $S_L^M(a)$.
To do so, we first recognize that,
for any fixed $a=(a_1,a_2,...,a_n)\in\R^n$,
there are $N$ terms of $S_L^M(a)$
in the expansion of $\sigma_j(a)\sigma_k(a)$,
if and only if
there are $N$ terms of $a_1^2a_2^2...a_M^2a_{M+1}...a_{L}$
in the expansion of $\sigma_j(a)\sigma_k(a)$.
But we know that,
for each $a_1^2a_2^2...a_M^2a_{M+1}...a_{L}$,
its divisor $a_1a_2...a_Ma_{i_1}a_{i_2}...a_{i_{j-M}}$
must come from the terms in the expansion of $\sigma_j(a)$,
where
$$i_1,i_2,...,i_{j-M}\in\set{M+1,M+2,...,L}$$
and are different from each other.
This indicates that the number of $S_L^M(a)$
in the expansion of $\sigma_j(a)\sigma_k(a)$
is exactly $C_{L-M}^{j-M}$.
(Note that we can also consider the divisor
\[a_1a_2...a_Ma_{s_1}a_{s_2}...a_{s_{k-M}}
=\frac{a_1^2a_2^2...a_M^2a_{M+1}...a_{L}}
{a_1a_2...a_Ma_{i_1}a_{i_2}...a_{i_{j-M}}}\]
of $a_1^2a_2^2...a_M^2a_{M+1}...a_{L}$,
comes from the terms in the expansion of $\sigma_k(a)$,
to yield a result $C_{L-M}^{k-M}$.
But since $j+k=L+M$ and hence $(j-M)+(k-M)=j+k-2M=L-M$,
we actually have
$C_{L-M}^{j-M}=C_{L-M}^{k-M}$.
Thus, both of these two observations are equivalent.)
Therefore, the coefficient of $S_{j+k-h}^{h}$
in the decomposition of $\sigma_j(a)\sigma_k(a)$ is
\[C_{(j+k-h)-h}^{j-h}=C_{j+k-2h}^{j-h},\]
and the coefficient of $S_{n-h}^{j+k-n+h}$
in the decomposition of $\sigma_j(a)\sigma_k(a)$ is
\[C_{(n-h)-(j+k-n+h)}^{j-(j+k-n+h)}
=C_{2n-j-k-2h}^{n-k-h}=C_{2n-j-k-2h}^{n-j-h}.\]
Thus the proof of \eref{eqn.sjkl} and \eref{eqn.sjkg} is complete.

\medskip

\emph{Step 4.}\quad
Since
\[(p-q)+(p+q+1)=2p+1\lessgtr n
~\Llra~p\lessgtr(n-1)/2,\]
substituting \eref{eqn.sjkl} and \eref{eqn.sjkg}
into \eref{eqn.sumqgtO}, we get
\begin{enumerate}[\quad(a)]
\item
For each $p\leq(n-1)/2$,
\begin{eqnarray*}
&&\sum_{q=0}^{p}{(-1)^q(2q+1)\sigma_{p-q}\sigma_{p+q+1}}\\
&=&\sum_{q=0}^{p}\sum_{h=0}^{p-q}
{(-1)^q(2q+1){C_{2p+1-2h}^{p-q-h}S_{2p+1-h}^{h}}}\\
&=&\sum_{q=0}^{p}\sum_{h=0}^{p}
{(-1)^q(2q+1){C_{2p+1-2h}^{p-q-h}S_{2p+1-h}^{h}}}\\
&=&\sum_{h=0}^{p}\left(\sum_{q=0}^{p}
{(-1)^q(2q+1)C_{2(p-h)+1}^{(p-h)-q}}\right){S_{2p+1-h}^{h}}\\
&=&\sum_{h=0}^{p}\left(\sum_{q=0}^{p-h}
{(-1)^q(2q+1)C_{2(p-h)+1}^{(p-h)-q}}\right){S_{2p+1-h}^{h}}\\
&=&S_{p+1}^{p}>0,
\end{eqnarray*}
where we used
$C_{2p+1-2h}^{p-q-h}\equiv0$ $(\forall h>p-q)$,
$C_{2(p-h)+1}^{(p-h)-q}\equiv0$ $(\forall q>p-h)$
and \eref{eqn.QIO} in the second,
fourth and fifth equality, respectively.
\item
For each $p>(n-1)/2$,
\begin{eqnarray*}
&&\sum_{q=0}^{n-1-p}{(-1)^q(2q+1)\sigma_{p-q}\sigma_{p+q+1}}\\
&=&\sum_{q=0}^{n-1-p}\sum_{h=0}^{n-1-p-q}
{(-1)^q(2q+1){C_{2n-2p-1-2h}^{n-p+q-h}S_{n-h}^{2p+1-n+h}}}\\
&=&\sum_{q=0}^{n-1-p}\sum_{h=0}^{n-1-p-q}
{(-1)^q(2q+1){C_{2n-2p-1-2h}^{n-1-p-q-h}S_{n-h}^{2p+1-n+h}}}\\
&=&\sum_{q=0}^{n-1-p}\sum_{h=0}^{n-1-p}
{(-1)^q(2q+1){C_{2n-2p-1-2h}^{n-1-p-q-h}S_{n-h}^{2p+1-n+h}}}\\
&=&\sum_{h=0}^{n-1-p}\left(\sum_{q=0}^{n-1-p}
{(-1)^q(2q+1)C_{2n-2p-1-2h}^{n-1-p-h-q}}\right){S_{n-h}^{2p+1-n+h}}\\
&=&\sum_{h=0}^{n-1-p}
\left(\sum_{q=0}^{n-1-p-h}
{(-1)^q(2q+1)C_{2(n-1-p-h)+1}^{(n-1-p-h)-q}}\right)
{S_{n-h}^{2p+1-n+h}}\\
&=&S_{p+1}^{p}>0,
\end{eqnarray*}
where we used
$C_{2n-2p-1-2h}^{n-1-p-q-h}\equiv0$ $(\forall h>n-1-p-q)$,
$C_{2n-2p-1-2h}^{n-1-p-h-q}
\equiv0$ $(\forall q>n-1-p-h)$
and \eref{eqn.QIO} in the third,
fifth and sixth equality, respectively.
\end{enumerate}
Combining \textsl{(a), (b)} and \eref{eqn.pwhn}, we obtain
\begin{equation*}
\wh Z_\ast(a)=\sum_{p=0}^{n-1}{S_{p+1}^{p}(a)},
\end{equation*}
that is, the desired formula \eref{eqn.wZ=SS}.
Thus the proof is completed.
\end{proof}

\medskip

\begin{remark}
\begin{enumerate}[(1)]
\item
The formula \eref{eqn.wZ=SS} in \lref{lem.A.XhY-YhX=SS}
gives us an elementary way to prove
the following two important inequalities in this paper:
\begin{eqnarray*}
\wh Z_\ast(a)
&\triangleq&\big(1-\sigma_2(a)+\sigma_4(a)-...\big)
\big(\sigma_1(a)-3\sigma_3(a)+5\sigma_5(a)-...\big)\\
&&\quad-\big(-2\sigma_2(a)+4\sigma_4(a)-...\big)
\big(\sigma_1(a)-\sigma_3(a)+\sigma_5(a)-...\big)\\&>&0,
\end{eqnarray*}
and
\begin{eqnarray*}
\wh Z(a)&\triangleq&
\cos\Theta\,\big(\sigma_1(a)-3\sigma_3(a)+5\sigma_5(a)-...\big)\\
&&\qquad-\sin\Theta\,\big(-2\sigma_2(a)+4\sigma_4(a)-...\big)
>0,~\mbox{if}~H(a)=\Theta.
\end{eqnarray*}
(See \cref{cor.XhY-YhX>0=>CY-SX>0} and \cref{cor.xikcksk>0}.)

\item
For the special case that $a=\mathds{1}$,
\eref{eqn.wZ=SS} in \lref{lem.A.XhY-YhX=SS} is exactly
\begin{eqnarray*}
\wh Z_\ast(\mathds{1})
&=&(1-C_n^2+C_n^4-...)
(C_n^1-3C_n^3+5C_n^5-...)\\
&&\quad-(-2C_n^2+4C_n^4-...)
(C_n^1-C_n^3+C_n^5-...)\\
&=&\sum_{0\leq 2i\leq n}(-1)^iC_n^{2i}
\sum_{0\leq 2j+1\leq n}(-1)^j(2j+1)C_n^{2j+1}\\
&&\quad-\sum_{0\leq 2i\leq n}(-1)^i(2i)C_n^{2i}
\sum_{0\leq 2j+1\leq n}(-1)^jC_n^{2j+1}\\
&=&\sum_{p=0}^{n-1}S_{p+1}^{p}(\mathds{1})
=\sum_{p=0}^{n-1}C_{n}^{p+1}C_{p+1}^{p}
=\sum_{p=0}^{n-1}C_{n}^{p+1}C_{p+1}^{1}\\
&=&\sum_{p=0}^{n-1}(p+1)C_{n}^{p+1}
=\sum_{k=0}^{n}k{C_n^k}
=n2^{n-1},
\end{eqnarray*}
where for the last equality
we just used the following standard tricks
\begin{eqnarray*}
J(1+x)^{J-1}&=&\frac{d}{dx}\left((1+x)^{J}\right)\\
&=&\frac{d}{dx}\left(\sum_{j=0}^{J}{C_{J}^{j}x^{j}}\right)
=\sum_{j=0}^{J}{{j}C_{J}^{j}x^{j-1}}.
\end{eqnarray*}
\end{enumerate}
\end{remark}

\medskip

\noindent\textbf{Acknowledgement.}
I thank my advisor Prof. D.-S. Li
for his encouragement and helpful conversations.
I also thank Prof. Y. Yuan
for useful discussions during the preparation of this paper.


\end{document}